\documentclass[11pt]{article}
\usepackage{url,fancybox,bm}
\usepackage{siunitx}
\usepackage{booktabs}       
\usepackage{amsfonts}       
\usepackage{nicefrac}       
\usepackage{microtype}
\usepackage{graphicx} 
\usepackage{graphics}
\usepackage{algorithm}
\usepackage{algorithmic}

\usepackage{amsmath}
\usepackage{mathrsfs}
\usepackage{amssymb}
\usepackage{amsthm}
\usepackage{color}
\usepackage[linktocpage,pagebackref,colorlinks,linkcolor=blue,anchorcolor=blue,citecolor=blue,urlcolor=blue]{hyperref}
\usepackage[left=1in,top=1in,right=1in,bottom=1in,letterpaper]{geometry}
\usepackage{subfigure}
\usepackage{footnote}

\usepackage{multirow}
\usepackage{bigstrut}
\usepackage{longtable}
\usepackage{cleveref}

\numberwithin{equation}{section}
\newtheorem{theorem}{Theorem}[section]

\newtheorem{proposition}[theorem]{Proposition}
\newtheorem{definition}[theorem]{Definition}

\newcommand{\M}{\mathcal{M}}

\newcommand{\R}{\mathbb{R}}

\newcommand{\CS}{\mathcal{S}}

\newcommand{\bx}{\mathbf{x}}

\newcommand{\etal}{ et al. }
\newcommand{\argmin}{\mathop{\rm argmin}}
\newcommand{\br}{\mathbb{R}}

\newcommand{\Proj}{{\rm Proj}}

\newcommand{\st}{\mathrm{s.t. }}

\newcommand{\be}{\begin{equation}}
\newcommand{\ee}{\end{equation}}
\newcommand{\ba}{\begin{array}}
\newcommand{\ea}{\end{array}}
\newcommand{\bad}{\begin{aligned}}
\newcommand{\ead}{\end{aligned}}
\newcommand{\normone}[1]{\| #1 \|_1}
\newcommand{\normtwo}[1]{\| #1 \|}

\newcommand{\inp}[2]{\langle #1, #2 \rangle}
\newcommand{\setword}[2]{\phantomsection #1\def\@currentlabel{\unexpanded{#1}}\label{#2}}

\newcommand{\mprox}{\mathrm{mprox}}
\newcommand{\prox}{\mathrm{prox}}

\newcommand{\calC}{\mathcal{C}}

\newcommand{\dist}{\operatorname{dist}}

\begin{document}
\title{Manifold Proximal Point Algorithms for Dual Principal Component Pursuit and Orthogonal Dictionary Learning\thanks{A short version of this paper appeared in the 53rd Annual Asilomar Conference on Signals, Systems, and Computers, Nov. 3--6, 2019.}}
\author{Shixiang Chen\thanks{Department of Industrial \& Systems Engineering, Texas A\&M University, College Station, TX, USA. Email: sxchen@tamu.edu}
\and Zengde Deng\thanks{Cainiao Network, Hangzhou, China. Email: dengzengde@gmail.com}
\and Shiqian Ma\thanks{Department of Mathematics, University of California, Davis, CA, USA. Research supported in part by NSF grants DMS-1953210 and CCF-2007797, and UC Davis CeDAR Innovative Data Science Seed Funding Program. Email: sqma@ucdavis.edu}
\and Anthony Man-Cho So\thanks{Department of Systems Engineering and Engineering Management, The Chinese University of Hong Kong, Hong Kong, China. Research supported in part by the Hong Kong RGC GRF Project CUHK 14203218 and in part by the CUHK Research Sustainability of Major RGC Funding Schemes Project 3133236. Email: manchoso@se.cuhk.edu.hk}
}
\date{\today}
\maketitle

\begin{abstract}
We consider the problem of {minimizing} the $\ell_1$ norm of a linear map over the sphere, which arises in various machine learning applications such as orthogonal dictionary learning (ODL) and robust subspace recovery (RSR). The problem is numerically challenging due to its nonsmooth objective and nonconvex constraint, and its algorithmic aspects have not been well explored. In this paper, we show how the manifold structure of the sphere can be exploited to design fast algorithms with provable guarantees for tackling this problem. Specifically, our contribution is fourfold. First, we present a manifold proximal point algorithm (ManPPA) for the problem and show that it converges at a global sublinear rate. Furthermore, we show that ManPPA can achieve a local quadratic convergence rate when applied to sharp instances of the problem. Second, we develop a semismooth Newton-based inexact augmented Lagrangian method for computing the search direction in each iteration of ManPPA and show that it has an asymptotic superlinear convergence rate. Third, we propose a stochastic variant of ManPPA called StManPPA, which is well suited for large-scale computation, and establish its sublinear convergence rate. Both ManPPA and StManPPA have provably faster convergence rates than existing subgradient-type methods. Fourth, using ManPPA as a building block, we propose a new heuristic method for solving a matrix analog of the problem, in which the sphere is replaced by the Stiefel manifold. The results from our extensive numerical experiments on the ODL and RSR problems demonstrate the efficiency and efficacy of our proposed methods.
\end{abstract}

\section{Introduction}
The problem of finding a subspace that captures the features of a given dataset and possesses certain properties is at the heart of many machine learning applications. One commonly encountered formulation of the problem, which is motivated largely by sparsity or robustness considerations, is given by
\be\label{DPCP}
\min_{\bm{x}\in\br^n} \ f(\bm{x}):=\|\bm{Y}^\top \bm{x}\|_1 \quad \st \quad \|\bm{x}\|_2 = 1,
\ee
where $\bm{Y}\in\br^{n\times p}$ is a given matrix and $\|\cdot\|_r$ denotes the $\ell_r$ norm of a vector. To better understand how problem~\eqref{DPCP} arises in applications, let us consider two representative examples.
\begin{itemize}
\item {\bf Orthogonal Dictionary Learning (ODL).} The goal of ODL is to find an orthonormal basis that can compactly represent a given set of $p$ ($p\gg n$) data points $\bm{y}_1,\ldots,\bm{y}_p \in \br^n$. Such a problem arises in many signal and image processing applications; see, e.g.,~\cite{bruckstein2009,rubinstein2010} and the references therein. By letting $\bm{Y} = \left[ \bm{y}_1, \ldots, \bm{y}_p \right] \in \br^{n\times p}$, the problem can be understood as finding an orthogonal matrix $\bm{X}\in\br^{n\times n}$ and a sparse matrix $\bm{A}\in\br^{n\times p}$ such that $\bm{Y} \approx \bm{X}\bm{A}$. Noting that this means $\bm{X}^\top\bm{Y} \approx \bm{A}$ should be sparse,
one approach is to find a collection of sparse vectors from the row space of $\bm{Y}$ and apply some post-processing procedure to the collection to form the orthogonal matrix $\bm{X}$. This has been pursued in various works; see, e.g.,~\cite{Spielman-Wang-Wright-2012,Qu-Sun-Wright-sparse-vector-2016,Sun-CDR-part1-2017,bai2018subgradient}. In particular, the work~\cite{bai2018subgradient} considers the formulation~\eqref{DPCP} and shows that under a standard generative model of the data, one can recover $\bm{X}$ from certain local minimizers of problem~\eqref{DPCP}.

\item {\bf Robust Subspace Recovery (RSR).} RSR is a fundamental problem in machine learning and data mining~\cite{Lerman-PIEEE-survey}. It is concerned with fitting a linear subspace to a dataset corrupted by outliers. Specifically, given a dataset $\bm{Y} = [ \bm{X}, \bm{O} ] \bm{\Gamma} \in \br^{n\times (p_1+p_2)}$, where the columns of $\bm{X} \in \br^{n\times p_1}$ are the inlier points spanning a $d$-dimensional subspace $\CS$ of $\br^n$ ($d<p_1$), the columns of $\bm{O} \in \br^{n\times p_2}$ are outlier points without a linear structure, and $\bm{\Gamma} \in \br^{(p_1+p_2)\times(p_1+p_2)}$ is an unknown permutation, the goal is to recover the inlier subspace $\CS$, or equivalently, to cluster the points into inliers and outliers. One recently proposed approach for solving this problem is the so-called dual principal component pursuit (DPCP)~\cite{Tsakiris-Vidal-2018,zhu2018dualpcp}. A key task in DPCP is to find a hyperplane that contains all the inliers. Such a task can be tackled by solving problem~\eqref{DPCP}. In fact, it has been shown in~\cite{Tsakiris-Vidal-2018,zhu2018dualpcp} that under certain conditions on the inliers and outliers, any global minimizer of problem~\eqref{DPCP} yields a vector that is orthogonal to the inlier subspace $\CS$.
\end{itemize}

Despite its attractive theoretical properties in various applications, problem~\eqref{DPCP} is numerically challenging to solve due to its nonsmooth objective and nonconvex constraint. Nevertheless, the manifold structure of the constraint set $\M:=\{\bm{x}\in\br^n \mid \|\bm{x}\|_2 = 1\}$ suggests that problem~\eqref{DPCP} could be amenable to manifold optimization techniques~\cite{Absil2009}. One approach is to apply smoothing to the nonsmooth objective in~\eqref{DPCP} and use existing algorithms for Riemannian smooth optimization to solve the resulting problem. For instance, when tackling the ODL problem, Sun\etal \cite{Sun-CDR-part1-2017,Sun-CDR-part2-2017} and Gilboa\etal \cite{Gilboa-2018} proposed to replace the absolute value function $t\mapsto|t|$ by the smooth surrogate $t\mapsto h_\mu(t)=\mu\log(\cosh(t/\mu))$ with $\mu>0$ being a smoothing parameter, while Qu\etal\cite{qu2020geometric} proposed to replace the $\ell_1$ norm with the $\ell_4$ norm. They then solve the resulting smoothed problems by either the Riemannian trust-region method \cite{Sun-CDR-part1-2017,Sun-CDR-part2-2017} or the Riemannian gradient descent method~\cite{Gilboa-2018,qu2020geometric}. Although it can be shown that these methods will yield the desired orthonormal basis under a standard generative model of the data, the smoothing approach can introduce significant analytic and computational difficulties~\cite{bai2018subgradient}. Another approach, which avoids smoothing the objective, is to solve \eqref{DPCP} directly using Riemannian nonsmooth optimization techniques. For instance, in the recent work~\cite{bai2018subgradient}, Bai\etal proposed to solve~\eqref{DPCP} using the Riemannian subgradient method (RSGM), which generates the iterates via
\be\label{Bai-RSM}
\bm{x}^{k+1} = \frac{\bm{x}^{k}-\eta_k \bm{v}^k}{\|\bm{x}^{k}-\eta_k \bm{v}^k\|_2}, \quad  \bm{v}^k\in\partial_R f(\bm{x}^k).
\ee
Here, $\eta_k>0$ is the step size; $\partial_R f(\cdot)$ denotes the Riemannian subdifferential of $f$ and is given by
\[ \partial_R f(\bm{x}) = (\bm{I}_n - \bm{x}\bm{x}^\top)\partial f(\bm{x}), \quad\forall\,\bm{x}\in\M, \]
where $\bm{I}_n$ is the $n\times n$ identity matrix and $\partial f(\cdot)$ is the usual subdifferential of the convex function $f$~\cite[Section 5]{Yang-manifold-optimality-2014}. Bai\etal\cite{bai2018subgradient} showed that for the ODL problem, RSGM with a suitable initialization will converge at a sublinear rate to a basis vector with high probability under a standard generative model of the data. Moreover, by running RSGM $O(n\log n)$ times, each time with an independent random initialization, one can recover the entire orthonormal basis with high probability. Around the same time, Zhu\etal\cite{zhu2018dualpcp} proposed a projected subgradient method (PSGM) for solving~\eqref{DPCP}. The method generates the iterates via
\be\label{PSGM}
\bm{x}^{k+1} = \frac{\bm{x}^{k}-\eta_k \bm{v}^k}{\|\bm{x}^{k}-\eta_k \bm{v}^k\|_2}, \quad \bm{v}^k\in \partial f(\bm{x}^k).
\ee
The updates \eqref{Bai-RSM} and \eqref{PSGM} differ in the choice of the direction $\bm{v}^k$---the former uses a \emph{Riemannian} subgradient of $f$ at $\bm{x}^k$, while the latter uses a usual \emph{Euclidean} subgradient. For the DPCP formulation of the RSR problem, Zhu\etal \cite{zhu2018dualpcp} showed that under certain assumptions on the data, PSGM with suitable initialization and piecewise geometrically diminishing step sizes will converge at a linear rate to a vector that is orthogonal to the inlier subspace $\CS$. The step sizes take the form $\eta_k = \eta^{\lfloor (k-K_0)/K \rfloor + 1}$, where $\eta\in(0,1)$ and $K_0,K\ge1$ satisfy certain conditions. In practice, however, the parameters $\eta,K_0,K$ are difficult to determine. Therefore, Zhu\etal \cite{zhu2018dualpcp} also proposed a PSGM with modified backtracking line search (PSGM-MBLS), which works well in practice but has no convergence guarantee. 

\subsection{Motivations for this Work}
Although the results in~\cite{bai2018subgradient,zhu2018dualpcp} demonstrate, both theoretically and computationally, the efficacy of RSGM and PSGM for solving instances of~\eqref{DPCP} that arise from the ODL and RSR problems, respectively, two fundamental questions remain. First, while PSGM can be shown to achieve a \emph{linear} convergence rate on the DPCP formulation of the RSR problem~\cite{zhu2018dualpcp}, only a \emph{sublinear} convergence rate has been established for RSGM on the ODL problem~\cite{bai2018subgradient}. Given the similarity of the updates~\eqref{Bai-RSM} and~\eqref{PSGM}, it is natural to ask whether the slower convergence rate of RSGM is an artifact of the analysis or due to the inherent structure of the ODL problem. Second, the convergence analyses in~\cite{bai2018subgradient,zhu2018dualpcp} focus only on the ODL and RSR problems. In particular, they do not shed light on the performance of RSGM or PSGM when tackling general instances of problem~\eqref{DPCP}. It would be of interest to fill this gap by identifying or developing practically fast methods that have more general convergence guarantees, especially since different applications may give rise to instances of problem~\eqref{DPCP} with different structures. In a recent attempt to address these questions, Li\etal\cite{li2019nonsmooth} showed, among other things, that RSGM will converge at the sublinear rate of $\mathcal{O}(k^{-1/4})$ (here, $k$ is the iteration counter) when applied to a general instance of problem~\eqref{DPCP} and at a linear rate when applied to a so-called \emph{sharp} instance of problem~\eqref{DPCP}. Informally, an optimization problem is said to possess the sharpness property if the objective function grows linearly with the distance to a set of local minima~\cite{burke1993weak}. Such a property plays a crucial role in establishing fast convergence guarantees for a host of iterative methods; see, e.g.,~\cite{burke1993weak,li2019nonsmooth,li2020nonconvex} and also~\cite{LYS17,ZS17,liu2019quadratic} for related results. Since the ODL problem and the DPCP formulation of the RSR problem are known to possess the sharpness property under certain assumptions on the data~\cite{bai2018subgradient,zhu2018dualpcp}, the results in~\cite{li2019nonsmooth} imply that RSGM will converge linearly on these problems.

\subsection{Our Contributions} \label{subsec:contrib}
In this paper, we depart from the subgradient-type approaches (such as RSGM~\eqref{Bai-RSM} and PSGM~\eqref{PSGM}) and present another method called the manifold proximal point algorithm (ManPPA) to tackle problem~\eqref{DPCP}. At each iterate $\bm{x}^k$, ManPPA computes a search directon by minimizing the sum of $f$ and a proximal term defined in terms of the \emph{Euclidean} distance over the tangent space to $\M$ at $\bm{x}^k$. This should be contrasted with other existing PPAs on manifolds (see, e.g.,~\cite{ferreira2002proximal,bento2017iteration}), in which the proximal term is defined in terms of the \emph{Riemannian} distance. Such a difference is important. Indeed, although the search direction defined in ManPPA does not admit a closed-form formula, it can be computed in a highly efficient manner by exploiting the structure of problem~\eqref{DPCP}; see Section~\ref{subsec:iALM}. However, the search direction defined in the existing PPAs on manifolds can be as difficult to compute as a solution to the original problem. Consequently, the applicability of those methods is rather limited.

We now summarize our contributions as follows:
\begin{enumerate}
\item We show that ManPPA has a global sublinear convergence rate of $\mathcal{O}(k^{-1/2})$ when applied to a general instance of problem~\eqref{DPCP}. Moreover, we show that if the instance has the sharpness property, then the local convergence rate of ManPPA is at least quadratic. Although the sublinear rate result follows from the results in~\cite{chen2018proximal}, the quadratic rate result is new. Moreover, both rates are superior to those of RSGM established in~\cite{li2019nonsmooth}. Key to the proof of the quadratic rate result is a new \emph{Riemannian subgradient inequality} (see Appendix~\ref{app:conv-anal-local}, Proposition \ref{weakly-inequality}), which extends the classic \emph{subgradient inequality} in the Euclidean space to the sphere $\M$. Such an inequality allows us to analyze ManPPA in a similar way as its Euclidean counterpart. It can also be of independent interest.

\item To compute the search direction in each iteration of ManPPA, we develop a semismooth Newton (SSN)-based inexact augmented Lagrangian method (ALM). Numerically, the proposed method can accurately compute the search direction in a highly efficient manner, which is crucial to the fast convergence of ManPPA. Theoretically, we show, for the first time, that the proposed SSN-based inexact ALM has an asymptotic superlinear convergence rate when finding the search direction.

\item We propose a stochastic version of ManPPA called StManPPA to tackle problem~\eqref{DPCP}. StManPPA is well suited for the setting where the number of the data points $p$ is extremely large, as each iteration involves only a simple closed-form update. We also analyze the convergence behavior of StManPPA. In particular, we show that it converges at the sublinear rate of $\mathcal{O}(k^{-1/4})$ when applied to a general instance of problem~\eqref{DPCP}, which matches the convergence rate of RSGM established in~\cite{li2019nonsmooth}. {Again, the aforementioned Riemannian subgradient inequality plays an important role in establishing this result, as it connects the analysis of StManPPA to those of various Euclidean stochastic methods.}

\item Using ManPPA as a building block, we develop a new method for solving the following matrix analog of problem~\eqref{DPCP}:
\be\label{DPCP-matrix}
    \min_{\bm{X} \in \R^{n\times q}} \normone{\bm{Y}^\top\bm{X}}\quad \st\quad \bm{X}^\top \bm{X} = \bm{I}_q.
\ee
Our interest in problem~\eqref{DPCP-matrix} stems from the observation that it provides alternative formulations of the ODL and RSR problems. Indeed, for the ODL problem, one can recover the entire orthonormal basis all at once by solving problem~\eqref{DPCP-matrix} with $q=n$. For the RSR problem, if one knows the dimension $d$ of the inlier subspace $\CS$, then one can recover it by solving problem~\eqref{DPCP-matrix} with $q=n-d$. We show that a good feasible solution to problem~\eqref{DPCP-matrix} can be found in a column-by-column manner by suitably modifying ManPPA. Although the proposed method is only a heuristic, our extensive numerical experiments show that it yields solutions of comparable quality to but is significantly faster than existing methods on the ODL and RSR problems.
\end{enumerate}

\subsection{Organization and Notation}
The rest of the paper is organized as follows. In Section \ref{sec:ManPPA}, we present ManPPA for solving problem \eqref{DPCP} and describe a highly efficient method for solving the subproblem that arises in each iteration of ManPPA. We also analyze the convergence behavior of ManPPA. In Section \ref{sec:Sto-ManPPA}, we propose StManPPA, a stochastic version of ManPPA that is well suited for large-scale computation, and analyze its convergence behavior. In Section \ref{sec:sequent} we discuss an extension of ManPPA for solving the matrix analog~\eqref{DPCP-matrix} of problem~\eqref{DPCP}. In Section \ref{sec:numerical}, we apply ManPPA to solve the ODL problem and the DPCP formulation of the RSR problem and compare its performance with some existing methods. We draw our conclusions in Section \ref{sec:conclusion}. 

Besides the notation introduced earlier, we use $L$ to denote the Lipschitz constant of $f$; i.e., $|f(\bm{x})-f(\bm{y})| \le L \|\bm{x}-\bm{y}\|_2$ for all $\bm{x},\bm{y}\in\R^n$ (note that {$L \le \sqrt{n} \max_{\bm{u}\in\R^n} \|\bm{Y}^\top\bm{u}\|_1  / \|\bm{u}\|_1$}). Given a closed set $\calC\subseteq\R^n$, we use $\Proj_{\calC}(\bm{x})$ to denote the projection of $\bm{x}$ onto $\calC$ and $\dist(\bm{x},\calC):=\inf_{\bm{y}\in\calC} \|\bm{y}-\bm{x}\|_2$ to denote the distance between $\bm{x}$ and $\calC$. Given a proper lower semicontinuous function $h:\R^n\rightarrow\R\cup\{+\infty\}$, its proximal mapping is given by $\prox_h(\bm{x}) = \argmin_{\bm{w}\in\R^n} h(\bm{w}) + \frac{1}{2}\|\bm{w}-\bm{x}\|_2^2$. Given two vectors $\bm{x},\bm{y}\in\R^n$, we use $\inp{\bm{x}}{\bm{y}}$ or $\bm{x}^\top\bm{y}$ to denote their usual inner product. Other notation is standard.

\section{A Manifold Proximal Point Algorithm}\label{sec:ManPPA}
Since problem~\eqref{DPCP} is nonconvex, our goal is to compute a \emph{stationary point} of~\eqref{DPCP}, which is a point $\bar{\bm x}\in\M$ that satisfies the first-order optimality condition
\[ \bm{0} \in \partial_Rf(\bar{\bm x}) = (\bm{I}_n-\bar{\bm x}\bar{\bm x}^\top)\partial f(\bar{\bm x}) \]
(see~\cite{Yang-manifold-optimality-2014}). In the recent work~\cite{chen2018proximal}, Chen\etal considered the more general problem of minimizing the sum of a smooth function and a nonsmooth convex function over the Stiefel manifold and developed a manifold proximal gradient method (ManPG) for finding a stationary point of it. When specialized to solve problem~\eqref{DPCP}, the method generates the iterates via
\begin{equation}\label{ManPPA-dpcp-retraction}
\bm{x}^{k+1} = \Proj_{\M}(\bm{x}^k + \alpha_k\bm{d}^k) = \frac{\bm{x}^k +\alpha_k\bm{d}^k}{\|\bm{x}^k +\alpha_k\bm{d}^k\|_2},
\end{equation}
where the search direction $\bm{d}^k$ is given by
\begin{equation}\label{ManPPA-dpcp-sub}
\begin{array}{rccl}
\bm{d}^k &=& \displaystyle\argmin_{\bm{d} \in \R^n} & \displaystyle\|\bm{Y}^\top(\bm{x}^k+\bm{d})\|_1 + \frac{1}{2t}\normtwo{\bm{d}}_2^2 \\
& & \st  & \bm{d}^\top \bm{x}^k = 0
\end{array}
\end{equation}
and $\alpha_k>0$, $t>0$ are the step sizes. As the reader may readily recognize, without the constraint $\bm{d}^\top\bm{x}^k=0$, the subproblem~\eqref{ManPPA-dpcp-sub} is simply computing the proximal mapping of $f$ at $\bm{x}^k$ and coincides with the update of the classic proximal point algorithm (PPA)~\cite{Rockafellar-76}. The constraint $\bm{d}^\top\bm{x}^k=0$ in~\eqref{ManPPA-dpcp-sub}, which states that the search direction $\bm{d}$ should lie on the tangent space to $\M$ at $\bm{x}^k$, is introduced to account for the manifold constraint in problem~\eqref{DPCP} and ensures that the next iterate $\bm{x}^{k+1}$ achieves sufficient decrease in objective value. Motivated by the above discussion, we call the method obtained by specializing ManPG to the setting of problem~\eqref{DPCP} \emph{ManPPA} and present its details in Algorithm \ref{alg:manppa}.

\begin{algorithm}[h]
	\caption{ManPPA for Solving Problem \eqref{DPCP}}\label{alg:manppa}
	\begin{algorithmic}[1]
		\STATE{Input: $\bx^0 \in \M$, $\beta\in(0,1)$, $t>0$.}
		\FOR{$k=0,1,2,\ldots$}
		\STATE{Solve the subproblem \eqref{ManPPA-dpcp-sub} to obtain $\bm{d}^k$.} \label{line:subpb}
		\STATE{Let $j_k$ be the smallest nonnegative integer such that  
		\[ f(\Proj_{\M}(\bm{x}^k + \beta^{j_k} \bm{d}^k)) \leq f(\bm{x}^k) -\frac{\beta^{j_k}}{2t} \normtwo{\bm{d}^k}_2^2. \]} \label{line:ls}
         \vspace{-0.7\baselineskip}
		\STATE{Set $\bm{x}^{k+1}$ according to~\eqref{ManPPA-dpcp-retraction} with $\alpha_k = \beta^{j_k}$.}
		\ENDFOR
	\end{algorithmic}
\end{algorithm}

Naturally, ManPPA inherits the properties of ManPG established in~\cite{chen2018proximal}. However, due to the structure of problem~\eqref{DPCP}, many of the developments in~\cite{chen2018proximal} have to be refined when designing ManPPA. {In particular, the SSN method used by ManPG for finding the search direction in each iteration requires the computation of the proximal mapping of the nonsmooth part of the objective function. However, due to the presence of the matrix $\bm{Y}$, the objective function $f$ of problem~\eqref{DPCP} does not have an easily computable proximal mapping. As such, the SSN method proposed in~\cite{chen2018proximal} cannot efficiently solve the subproblem~\eqref{ManPPA-dpcp-sub}. To circumvent this difficulty, we propose to use an inexact ALM, which can efficiently compute an accurate solution to \eqref{ManPPA-dpcp-sub}; see Section~\ref{subsec:iALM}.}

Now, let us state the following result, which shows that the line search step in line~\ref{line:ls} of Algorithm~\ref{alg:manppa} is well defined. It simplifies~\cite[Lemma 5.2]{chen2018proximal} and yields sharper constants. The proof can be found in Appendix~\ref{app:suffdec}.
\begin{proposition}\label{prop:suff_decrease}
Let $\{(\bm{x}^k,\bm{d}^k)\}_k$ be the sequence generated by Algorithm \ref{alg:manppa}. Define $\bar{\alpha} = \min\{1,1/(tL)\}$. For any $\alpha \in (0,\bar{\alpha}]$, we have
\be\label{ineq:suff_decrease}
	f(\Proj_{\M}(\bm{x}^k + \alpha\bm{d}^k)) \leq f(\bm{x}^k) - \frac{\alpha}{2t}\normtwo{\bm{d}^k}_2^2.
\ee
As a result, we have $\alpha_k=\beta^{j_k} > \beta\bar{\alpha}$ for any $k\geq 0$ in Algorithm \ref{alg:manppa}, which implies that the line search step terminates after at most $\lceil\log_{\beta}\bar{\alpha}\rceil+1$ iterations.
In particular, if $t\leq 1/L$, then we have $\bar{\alpha} = 1$, which implies that we can take $j_k=0$ in line~\ref{line:ls} of Algorithm \ref{alg:manppa}; i.e., no line search is needed.
\end{proposition}

\subsection{Convergence Analysis of ManPPA}
In this subsection, we study the convergence behavior of ManPPA. Recall from~\cite[Lemma 5.3]{chen2018proximal} that if $\bm{d}^k=\bm{0}$ in~\eqref{ManPPA-dpcp-sub}, then $\bm{x}^k\in\M$ is a stationary point of problem~\eqref{DPCP}. This motivates us to call $\bm{x}^k\in\M$ an \emph{$\epsilon$-stationary point} of problem~\eqref{DPCP} with $\epsilon\ge0$ if the solution $\bm{d}^k$ to~\eqref{ManPPA-dpcp-sub} satisfies $\|\bm{d}^k\|_2 \le \epsilon$. By specializing the convergence results in~\cite[Theorem 5.5]{chen2018proximal} for ManPG to ManPPA, we obtain the following theorem:
\begin{theorem}\label{thm:global_convergence}
Any limit point of the sequence $\{\bm{x}^k\}_k$ generated by Algorithm \ref{alg:manppa} is a stationary point of problem \eqref{DPCP}. Moreover, Algorithm \ref{alg:manppa} with $t=1/L$ returns an $\epsilon$-stationary point $\bm{x}^k$ in at most $\lceil 2(f(\bm{x}^0)-f^*)/(L\epsilon^2) \rceil$ iterations, where $f^*$ is the optimal value of problem~\eqref{DPCP}.
\end{theorem}
Theorem~\ref{thm:global_convergence} shows that ManPPA has an iteration complexity of $\mathcal{O}(\epsilon^{-2})$, which is superior to the $\mathcal{O}(\epsilon^{-4})$ bound established for RSGM in~\cite{li2019nonsmooth}.

Now, let us analyze the convergence rate of ManPPA in the setting where problem~\eqref{DPCP} possesses the sharpness property. Such a setting is highly relevant in applications, as both the ODL problem and DPCP formulation of the RSR problem give rise to sharp instances of problem~\eqref{DPCP} under certain assumptions on the data; see~\cite[Proposition C.8]{bai2018subgradient} and \cite[Proposition 4]{li2019nonsmooth}. To proceed, we first introduce the notion of sharpness.
\begin{definition}[Sharpness; see, e.g., \cite{burke1993weak}]\label{def:sharpness2}
We say that $\mathcal{X} \subseteq \M$ is a set of \emph{weak sharp minima} for the function $f$ with parameters $(\alpha,\delta)$ (where $\alpha,\delta>0$) if for any $\bm{x} \in \mathcal{B}(\delta) := \{ \bm{x} \in \M \mid \dist(\bm{x},\mathcal{X}) \le \delta \}$, we have
\be\label{def:sharp_ineq}
f(\bm{x}) - f(\bar{\bm x}) \ge \alpha\cdot\dist(\bm{x},\mathcal{X}), \quad\forall \bar{\bm x} \in \mathcal{X}.
\ee
\end{definition}
From the definition, we see that if $\mathcal{X}$ is a set of weak sharp minima of $f$, then it is the set of minimizers of $f$ over $\mathcal{B}(\delta)$. Moreover, the function value grows linearly with the distance to $\mathcal{X}$. In the presence of such a regularity property, ManPPA can be shown to converge at a much faster rate. The following result, which has not appeared in the literature before and is thus new, constitutes the first main contribution of this paper. 
\begin{theorem}\label{thm:local_rate}
Suppose that $\mathcal{X}\subseteq\M$ is a set of weak sharp minima for the function $f$ with parameters $(\alpha,\delta)$. Let $\{\bm{x}^k\}_k$ be the sequence generated by Algorithm~\ref{alg:manppa} with $\dist(\bm{x}^0,\mathcal{X})< \overline{\delta}:=\min \left\{\delta, \tfrac{\alpha}{L} \right\}$ and $t\leq  \min \left\{ \tfrac{\overline{\delta}}{2\alpha-L \overline{\delta}}, \frac{2\overline{\delta} \alpha - L\overline{\delta}^2 }{L^2} \right\}$. Then, we have
\be\label{the-B5-proof-1}
\dist(\bm{x}^k,\mathcal{X})\leq \overline{\delta}, \quad \forall k\geq 0,
\ee
\be\label{the-B5-proof-2}
\dist(\bm{x}^{k+1} , \mathcal{X})\leq \mathcal{O}( \dist^2(\bm{x}^k,\mathcal{X})),\quad \forall k\geq 0.
\ee
\end{theorem}
Theorem~\ref{thm:local_rate} establishes the quadratic convergence rate of ManPPA when applied to a sharp instance of problem~\eqref{DPCP}. Again, this is superior to the linear convergence rate of RSGM established in~\cite{li2019nonsmooth} for this setting. The proof of Theorem~\ref{thm:local_rate} can be found in Appendix~\ref{app:conv-anal-local}. {Note that since $\M$ is nonconvex, one cannot directly apply standard convergence analysis techniques for PPA (see, e.g., \cite{Rockafellar-76}) to obtain Theorem~\ref{thm:local_rate}. The key to overcoming this diffculty is the new Riemannian subgradient inequality we establish in Proposition \ref{weakly-inequality} (see Appendix \ref{app:conv-anal-local}), which provides a path for extending the convergence analysis of PPA to that of ManPPA.}

It should be pointed out that the results in Theorems~\ref{thm:global_convergence} and~\ref{thm:local_rate} do not assume any generative model of the data matrix $\bm{Y}$. By contrast, the results developed in, e.g.,~\cite{bai2018subgradient} for the ODL problem and~\cite{zhu2018dualpcp} for the DPCP formulation of the RSR problem do assume certain generative models of the data. Although the latter results may yield qualitatively sharper convergence guarantees for instances of~\eqref{DPCP} that arise from the ODL problem or the DPCP formulation of the RSR problem, the former apply to arbitrary instances of \eqref{DPCP}.

\subsection{Solving the Subproblem \eqref{ManPPA-dpcp-sub}} \label{subsec:iALM}
Observe that each iteration of ManPPA requires solving the subproblem~\eqref{ManPPA-dpcp-sub} to obtain the search direction. Thus, the efficiency of ManPPA depends not only on its convergence rate (which has already been studied in Theorems~\ref{thm:global_convergence} and~\ref{thm:local_rate}) but also on how fast the subproblem~\eqref{ManPPA-dpcp-sub} can be solved. To address the latter, we note that  \eqref{ManPPA-dpcp-sub} is a linearly constrained strongly convex quadratic minimization problem. This motivates us to adopt the SSN-based ALM originally developed in~\cite{Sun-lasso-2018} for LASSO-type problems to solve it. As we shall see, such an approach yields a highly efficient method for solving the subproblem~\eqref{ManPPA-dpcp-sub}.

To set the stage for our development, let us drop the index $k$ from~\eqref{ManPPA-dpcp-sub} for simplicity and set $\bm{c} = \bm{Y}^\top\bm{x}$. Then, the subproblem \eqref{ManPPA-dpcp-sub} can be equivalently written as
\be\label{ManPPA-dpcp-sub-rewrite}
\min_{\bm{d}\in\R^n, \atop \bm{u}\in\R^p} \ \frac{1}{2}\|\bm{d}\|_2^2+t\|\bm{u}\|_1 \,\,\, \st \,\,\, \bm{Y}^\top\bm{d} + \bm{c} = \bm{u},\, \bm{d}^\top \bm{x} = 0.
\ee
At this point, one may be tempted to use ADMM to solve problem~\eqref{ManPPA-dpcp-sub-rewrite}. However, from a practical point of view, ADMM is often unable to return a high-accuracy solution in an efficient manner. Since \eqref{ManPPA-dpcp-sub-rewrite} is a subproblem in ManPPA, a low-accuracy solution will adversely affect the convergence rate of ManPPA. In fact, this has been observed in our numerical experiments. Therefore, we propose to use an inexact ALM, which can solve problem~\eqref{ManPPA-dpcp-sub-rewrite} efficiently and accurately. This makes it possible for ManPPA to achieve fast convergence. To describe the algorithm, let us first write down the augmented Lagrangian function corresponding to~\eqref{ManPPA-dpcp-sub-rewrite}:
\be \label{augmented_func}
\mathcal{L}_{\sigma}(\bm{d},\bm{u};y,\bm{z}) := \frac{1}{2}\|\bm{d}\|_2^2+t\|\bm{u}\|_1+ y\cdot \bm{d}^\top\bm{x} + \langle \bm{z}, \bm{Y}^\top\bm{d} + \bm{c} - \bm{u} \rangle + \frac{\sigma}{2}(\bm{d}^\top\bm{x})^2  +\frac{\sigma}{2} \| \bm{Y}^\top\bm{d} + \bm{c} - \bm{u}\|_2^2.
\ee
Here, $y\in\mathbb{R}$ and $\bm{z}\in\mathbb{R}^p$ are Lagrange multipliers (dual variables) associated with the constraints in~\eqref{ManPPA-dpcp-sub-rewrite} and $\sigma>0$ is a penalty parameter. Then, the inexact ALM for solving \eqref{ManPPA-dpcp-sub-rewrite} can be described as follows~\cite{rockafellar1976augmented,Sun-lasso-2018}:

\begin{algorithm}[ht]
	\caption{Inexact ALM for Solving Problem \eqref{ManPPA-dpcp-sub-rewrite}}
	\label{alg:augmented_vector}
	\begin{algorithmic}[1]
		\STATE{Input: $\bm{d}^0 \in \R^n$, $\bm{u}^0 \in \R^p$, $y^0 \in \R$, $\bm{z}^0 \in \R^p$, $\sigma_0>0$.}
		\FOR{$j=0,1,\dots$}
		\STATE Compute
		\begin{align} \label{subproblem:augmented}
		(\bm{d}^{j+1},\bm{u}^{j+1}) &\approx \argmin_{\bm{d}\in\R^n, \atop \bm{u}\in\R^p}  \Psi_j(\bm{d},\bm{u}) := \mathcal{L}_{\sigma_{j}}(\bm{d},\bm{u};y^{j},\bm{z}^{j}).
		\end{align}
		\vspace{-0.5\baselineskip}
		\STATE Update dual variables:
		\begin{align*}
		y^{j+1} &= y^{j}+\sigma_{j}(\bm{d}^{j+1})^{\top}\bm{x}, \\
\bm{z}^{j+1} &= \bm{z}^{j}+\sigma_{j}(\bm{Y}^\top\bm{d}^{j+1}+\bm{c}-\bm{u}^{j+1}).
		\end{align*}
		\vspace{-0.8\baselineskip}
        \STATE Update $\sigma_{j+1}\nearrow\sigma_{\infty}\leq+\infty$.
		\ENDFOR		
	\end{algorithmic}
\end{algorithm}

Since the subproblem~\eqref{subproblem:augmented} can only be solved inexactly in general, we adopt the following stopping criteria, which are standard in the literature (see~\cite{rockafellar1976augmented,Sun-lasso-2018}):
\begin{subequations}
\begin{align}
& \Psi_{j}(\bm{d}^{j+1},\bm{u}^{j+1}) - \Psi_{j}^* \leq \frac{\varepsilon^2_j}{2\sigma_{j}}, \ \sum_{j=0}^\infty\varepsilon_j < \infty, \label{augmented_subproblem_inexact_cond1}  \\
& \Psi_{j}(\bm{d}^{j+1},\bm{u}^{j+1}) - \Psi_{j}^* \leq \frac{\delta^2_j}{2\sigma_j} \| (y^{j+1},\bm{z}^{j+1})-(y^j,\bm{z}^j) \|_2^2, \ \sum_{j=0}^\infty\delta_j < \infty,  \label{augmented_subproblem_inexact_cond2} \\
& \dist(\mathbf{0},\partial\Psi_{j}(\bm{d}^{j+1},\bm{u}^{j+1})) \leq \frac{\delta_j^{\prime}}{\sigma_j}\| (y^{j+1},\bm{z}^{j+1})-(y^j,\bm{z}^j)\|_2, \ \delta_j^{\prime} \searrow 0. \label{augmented_subproblem_inexact_cond3}
\end{align}
\end{subequations}
Here, $\Psi_j^*$ is the optimal value of~\eqref{subproblem:augmented}. Conditions \eqref{augmented_subproblem_inexact_cond1}--\eqref{augmented_subproblem_inexact_cond3} ensure that starting from any initial point $(\bm{d}^0,\bm{u}^0;y^0,\bm{z}^0)$, the inexact ALM (Algorithm~\ref{alg:augmented_vector}) will converge at a superlinear rate to an optimal solution to problem~\eqref{ManPPA-dpcp-sub-rewrite}. This result, which constitutes the second main contribution of this paper, is obtained from  a new perturbation analysis of the solution set of the subproblem~\eqref{ManPPA-dpcp-sub} and its dual. The proof can be found in Appendix~\ref{app:ALM-SSN}.

Now, it remains to discuss how to solve the subproblem~\eqref{subproblem:augmented} in an efficient manner. Again, let us drop the index $j$ in~\eqref{subproblem:augmented} for simplicity. By simple manipulation, we have
\[
\Psi(\bm{d},\bm{u}) = \frac{1}{2}\|\bm{d}\|^2_2 + \frac{\sigma}{2}\left(\bm{d}^\top\bm{x}+\frac{y}{\sigma}\right)^2 - \frac{y^2}{2\sigma}-\frac{\|\bm{z}\|_2^2}{2\sigma} + t\|\bm{u}\|_1 + \frac{\sigma}{2} \left\| \bm{Y}^\top \bm{d} + \bm{c} + \frac{\bm{z}}{\sigma} - \bm{u} \right\|_2^2.
\]
Consider the function $\bm{d}\mapsto\psi(\bm{d}):=\inf_{\bm{u}\in\R^p} \Psi(\bm{d},\bm{u})$. Upon letting $\bm{w} = \bm{Y}^\top \bm{d} + \bm{c} + \frac{\bm{z}}{\sigma} \in \R^p$ and using the definition of the proximal mapping of $\bm{u}\mapsto h(\bm{u}):=t\|\bm{u}\|_1$, we have
\[
\psi(\bm{d}) = \frac{1}{2}\|\bm{d}\|^2_2 + \frac{\sigma}{2}\left(\bm{d}^\top\bm{x}+\frac{y}{\sigma}\right)^2 - \frac{y^2}{2\sigma}-\frac{\|\bm{z}\|_2^2}{2\sigma} + h(\prox_{h/\sigma}(\bm{w})) + \frac{\sigma}{2}\| \bm{w} - \prox_{h/\sigma}(\bm{w}) \|_2^2.
\]
It follows that $(\bar{\bm{d}},\bar{\bm{u}})=\argmin_{\bm{d}\in\R^n, \bm{u}\in\R^p} \ \Psi(\bm{d},\bm{u})$ if and only if
\[
\bar{\bm{d}} = \argmin_{\bm{d}\in\R^n} \ \psi(\bm{d}), \quad \bar{\bm{u}} = \prox_{h/\sigma} \left( \bm{Y}^\top\bar{\bm{d}}+\bm{c}+\frac{\bm{z}}{\sigma} \right).
\]
Using~\cite[Theorem 2.26]{RW04} and the Moreau decomposition $\bm{w} = \prox_{h/\sigma} (\bm{w}) + (1/\sigma) \prox_{\sigma h^*}(\sigma \bm{w})$, where $h^*$ is the conjugate function of $h$, it can be deduced that $\psi$ is strongly convex and continuously differentiable with
\begin{align*}
\nabla \psi(\bm{d})=\bm{d} + \sigma \left(\bm{d}^\top\bm{x} + \frac{y}{\sigma}\right)\bm{x} + \bm{Y}\prox_{\sigma h^{*}}(\sigma \bm{w}).
\end{align*}
Thus, we can find $\bar{\bm{d}}$ by solving the nonsmooth equation
\begin{align}\label{nonsmooth-equation}
\nabla\psi(\bm{d})=\bm{0}.
\end{align}
Towards that end, we apply an SSN method, which finds the solution by successive linearization of the map $\nabla\psi$. To implement the method, we first need to compute the generalized Jacobian of $\nabla\psi$~\cite[Definition 2.6.1]{C90}, denoted by $\partial(\nabla\psi)$. By the chain rule~\cite[Corollary of Theorem 2.6.6]{C90} and the Moreau decomposition, each element $\bm{V} \in \partial(\nabla\psi)$ takes the form
\begin{equation} \label{eq:ghess}
\bm{V} = \bm{I}_n + \sigma \bm{Y}(\bm{I}_p - \bm{Q})\bm{Y}^\top + \sigma \bm{x}\bm{x}^{\top},
\end{equation}
where $\bm{Q}\in\partial\prox_{h/\sigma}(\bm{w})$. Using the definition of $h$, it can be shown that the diagonal matrix $\bm{Q} = {\rm Diag}(\bm{q})$ with
\[
q_i = \left\{
\begin{array}{c@{\quad}l}
0 & \mbox{if } |w_i| \le t/\sigma, \\
\noalign{\smallskip}
1 & \mbox{otherwise},
\end{array} \quad i=1,\ldots,p
\right.
\]
is an element of $\partial\prox_{h/\sigma}(\bm{w})$~\cite[Section 3.3]{Sun-lasso-2018} and hence can be used to define an element $\bm{V} \in \partial(\nabla\psi)$ via~\eqref{eq:ghess}. Note that the matrix $\bm{V}$ so defined is positive definite. As such, the following generic iteration of the SSN method for solving~\eqref{nonsmooth-equation} is well defined:
\begin{subequations} \label{ssn}
\begin{align}
\bm{v} &= -\bm{V}^{-1} \nabla\psi(\bm{d}^j), \label{eq:ssn-a} \\
\bm{d}^{j+1} &= \bm{d}^j + \rho_j\bm{v}
\end{align}
\end{subequations}
Here, $\rho_j>0$ is the step size. Moreover, since $\bm{I}_p-\bm{Q}$ is a diagonal matrix whose entries are either 0 or 1, the matrix $\bm{V}$ can be assembled in a very efficient manner; again, see~\cite{Sun-lasso-2018}. The detailed implementation of the SSN method for solving~\eqref{nonsmooth-equation} is given in Algorithm~\ref{alg:SSN_vector}.

\begin{algorithm}
	\caption{SSN method for Solving the Nonsmooth Equation \eqref{nonsmooth-equation}}
	\label{alg:SSN_vector}
	\begin{algorithmic}[1]
		\STATE{Input: $\mu\in(0,1/2)$, $\bar{\eta}\in[0,1)$, $\tau\in(0,1]$, $\delta\in(0,1)$.}
		\FOR{$j=0,1,\dots$}
		\STATE Choose $\bm{Q}^{j}\in\partial\prox_{h/\sigma} \left( \bm{Y}^\top \bm{d}+\bm{c}+\tfrac{\bm{z}}{\sigma} \right)$. Let $\bm{V}^{j}=\bm{I}_n + \sigma \bm{Y}(\bm{I}_p -\bm{Q}^j)\bm{Y}^\top + \sigma \bm{x}\bm{x}^{\top}$. Find an approximate solution ${\bm v}^j$ to the linear system
		\begin{align*} 
		\bm{V}^j\bm{v} = -\nabla\psi(\bm{d}^j)
		\end{align*}
		that satisfies 
		\begin{align*}
		\|\bm{V}^j{\bm v}^j + \nabla\psi(\bm{d}^j)\|_2 \leq \min\left\{ \bar{\eta}, \|\nabla\psi(\bm{d}^j)\|_2^{1+\tau} \right\}.
		\end{align*}
		\vspace{-0.8\baselineskip}
		\STATE Let $m_j$ be the smallest nonnegative integer such that
		\begin{align*}
		\psi(\bm{d}^{j}+\delta^{m_j}\bm{v}^j) \leq \psi(\bm{d}^j)+\mu\delta^{m_j}\inp{\nabla\psi(\bm{d}^j)}{\bm{v}^j}.
		\end{align*}
		\vspace{-0.8\baselineskip}
		\STATE Set $\bm{d}^{j+1}=\bm{d}^j+\rho_j\bm{v}^j$ with $\rho_j = \delta^{m_j}$.
		\ENDFOR
	\end{algorithmic}
\end{algorithm}

The SSN method (Algorithm~\ref{alg:SSN_vector}) will converge at a superlinear rate to the unique solution $\bar{\bm d}$ to~\eqref{nonsmooth-equation}. The proof can be found in Appendix~\ref{app:ALM-SSN}.

\section{Stochastic Manifold Proximal Point Algorithm}\label{sec:Sto-ManPPA}
In this section, we propose, for the first time, a stochastic ManPPA (StManPPA) for solving problem \eqref{DPCP}, which is well suited for the setting where $p$ (typically representing the number of data points) is much larger than $n$ (typically representing the ambient dimension of the data points). To begin, observe that problem \eqref{DPCP} has the finite-sum structure
\[ 
\min_{\bm{x}\in\R^n} \ \sum_{j=1}^{p} \left| \bm{y}_j^\top\bm{x} \right| \quad \st \quad \|\bm{x}\|_2=1,
\] 
where $\bm{y}_j\in\R^n$ is the $j$-th column of $Y$. When $p$ is extremely large, computing the matrix-vector product $\bm{Y}^\top\bm{x}$ can be expensive. To circumvent this difficulty, in each iteration of StManPPA, a column of $\bm{Y}$, say $\bm{y}_j$, is randomly chosen and the search direction $\bm{d}^k$ is given by
\begin{equation}\label{subproblem_sppa}
\begin{array}{rccl}
\bm{d}^k &=& \displaystyle\argmin_{\bm{d} \in \R^n} & \left| \bm{y}_j^\top (\bm{x}^k+\bm{d}) \right| + \displaystyle\frac{1}{2t}\|\bm{d}\|_2^2 \\
& & \st  & \bm{d}^\top \bm{x}^k = 0.
\end{array}
\end{equation}
The key advantage of StManPPA is that the subproblem \eqref{subproblem_sppa} admits a closed-form solution that is very easy to compute.
\begin{proposition}\label{prop:sto-sol}
Let $\mu=t(\bm{y}_j^\top\bm{x}^k)$. Then, the solution to \eqref{subproblem_sppa} is given by
\begin{equation} \label{opt-d}
 \bm{d}^k = \left\{
 \begin{array}{cl}
 \mu \bm{x}^k - t\bm{y}_j & \text{if }(1+\mu)\mu/t - t\|\bm{y}_j\|_2^2>0,  \\
 \noalign{\smallskip}
 -\mu \bm{x}^k + t\bm{y}_j & \text{if } (1-\mu)\mu/t + t\| \bm{y}_j \|_2^2 < 0, \\
 \noalign{\smallskip}
 \displaystyle \frac{\mu^2\bm{x}^k - t\mu \bm{y}_j}{t^2\|\bm{y}_j\|_2^2-\mu^2} & \text{otherwise}.
 \end{array}
 \right.
\end{equation}
\end{proposition}
\begin{proof}
The first-order optimality conditions of \eqref{subproblem_sppa} are 
\begin{subequations} \label{eq:kkt-sppa}
\begin{align}
\label{subproblem_sppa-kkt-1}	\bm{0} &\in \frac{1}{t}\bm{d} + \partial \left| \bm{y}_j^\top (\bm{x}^k+\bm{d}) \right| \bm{y}_j - \lambda \bm{x}^k,\\
\label{subproblem_sppa-kkt-2}	0 &=\bm{d}^\top \bm{x}^k.
\end{align}
\end{subequations}
Suppose that $\bm{d}$ is a solution to~\eqref{eq:kkt-sppa}. If $\bm{y}_j^\top (\bm{x}^k+\bm{d})>0$, then $\partial \left| \bm{y}_j^\top (\bm{x}^k+\bm{d}) \right| = 1$, and \eqref{subproblem_sppa-kkt-1} implies that $\bm{0} = \bm{d}/t + \bm{y}_j - \lambda \bm{x}^k$. This, together with \eqref{subproblem_sppa-kkt-2} and the fact that $\|\bm{x}^k\|_2=1$, gives $\lambda=  \bm{y}_j^\top \bm{x}^k$. It follows that $\bm{d}=t (\bm{y}_j^\top \bm{x}^k) \bm{x}^k - t\bm{y}_j = \mu \bm{x}^k - t\bm{y}_j$ and hence $\bm{y}_j^\top (\bm{x}^k+\bm{d})>0$ is equivalent to $(1+\mu)\mu/t - t\|\bm{y}_j\|_2^2>0$. This establishes the first case in \eqref{opt-d}. The other two cases in \eqref{opt-d}, which correspond to $\bm{y}_j^\top (\bm{x}^k+\bm{d})<0$ and $\bm{y}_j^\top (\bm{x}^k+\bm{d})=0$, can be derived using a similar argument. 
\end{proof}

We now present the details of StManPPA in Algorithm~\ref{alg:stmanppa}. It is worth noting that our proposed StManPPA is different from the ones developed in the recent work~\cite{Wang-Ma-Xue-2020}. Indeed, in each iteration, the former only needs to solve a subproblem that involves a single component of the objective function, while the latter need to compute the proximal mapping of the entire objective function.
\begin{algorithm}
	\caption{StManPPA for Solving Problem~\eqref{DPCP}}
	\label{alg:stmanppa}
	\begin{algorithmic}[1]
		\STATE Input: $\bm{x}^0 \in \M$, $t_0,t_1,\ldots,t_T>0$.
		\FOR{$k=0,1,\dots,T$}
		\STATE Select $j_k\in\{1,\ldots,p\}$ uniformly at random and solve the subproblem~\eqref{subproblem_sppa} with $j=j_k$, $t=t_k$ to obtain $\bm{d}^k$.
		\STATE Set $\bm{x}^{k+1} = \Proj_{\M}(\bm{x}^k + \bm{d}^k)$.
		\ENDFOR
        \STATE Output: $\bar{\bm{x}}=\bm{x}^k$ with probability $t_k/\sum_{k=0}^T t_k$.
	\end{algorithmic}
\end{algorithm}

\subsection{Convergence Analysis of StManPPA}

In this section, we present our convergence results for StManPPA. Let us begin with some preparations. Define $f_j:\R^n\rightarrow\R$ to be the function $f_j(\bm{x}) = \left| \bm{y}_j^\top\bm{x} \right|$ and let $L_j>0$ denote the Lipschitz constant of $f_j$, where $j=1,\ldots,p$. Set $\bar{L} := \max_{j\in\{1,\ldots,p\}} L_j$. Furthermore, define the \emph{Moreau envelope} and \emph{proximal mapping} on $\M$ by
\begin{subequations}\label{surrogate_measure}
\begin{align}
e_f(\bm{z}) &= \min_{\bm{x}\in\M} \ f(\bm{x}) + \frac{1}{2}\| \bm{x} - \bm{z} \|_2^2, \label{eq:moreau} \\
\mprox_{f}(\bm{z}) &\in \argmin_{\bm{x}\in\M} \ f(\bm{x}) + \frac{1}{2}\| \bm{x} - \bm{z} \|_2^2,
\end{align}
\end{subequations}
respectively. The proximal mapping $\mprox$ is well defined since the constraint set $\M$ is compact. As it turns out, the proximal mapping $\mprox$ can be used to define an alternative notion of stationarity for problem~\eqref{DPCP}. Indeed, for any $\lambda>0$ and $\bm{x}\in\M$, if we denote $\hat{\bm x} = \mprox_{\lambda f}(\bm{x})$, then the optimality condition of~\eqref{surrogate_measure} yields
\[
\bm{0} \in \partial_R f(\hat{\bm x}) + \frac{1}{\lambda}(\bm{I}_n-\hat{\bm x} \hat{\bm x}^{\top}) (\hat{\bm x} -\bm{x}).
\]
Since $\bm{I}_n-\hat{\bm x} \hat{\bm x}^{\top}$ is a projection operator and hence nonexpansive, we obtain
\[ \dist(\bm{0},\partial_R f(\hat{\bm x})) \leq \frac{1}{\lambda} \| \bm{x} - \hat{\bm x} \|_2. \]
In particular, if $\tfrac{1}{\lambda}\| \bm{x} - \hat{\bm x} \|_2 \leq \epsilon$, then (i) $\hat{\bm x}$ is $\epsilon$-stationary in the sense that $\dist(\bm{0},\partial_R f(\hat{\bm x})) \le \epsilon$ and (ii) $\bm{x}$ is close to the $\epsilon$-stationary point $\hat{\bm x}$. This motivates us to use 
\[ \M \ni \bm{x} \mapsto \Theta_\lambda(\bm{x}) := \frac{1}{\lambda} \|\bm{x} - \mprox_{\lambda f}(\bm{x})\|_2 \]
as a stationarity measure for problem~\eqref{DPCP}. We call $\bm{x}\in\M$ an \emph{$\epsilon$-nearly stationary point} of problem~\eqref{DPCP} if $\Theta_\lambda(\bm{x}) \le \epsilon$. It is worth noting that such a notion has also been used in~\cite{li2019nonsmooth} to study the stochastic RSGM.

We are now ready to establish the convergence rate of StManPPA, which constitutes the third main contribution of this paper. 
\begin{theorem}\label{prop:stochastic subgradient global rate}
For any $\lambda\in(0,1/(p\bar{L}))$, the point $\bar{\bm x}$ output by Algorithm~\ref{alg:stmanppa} satisfies 
\[ 
   \mathbb{E} \left[ \Theta_\lambda(\bar{\bm x})^2 \right]  \leq \frac{ 2\lambda e_{\lambda}( \bm{x}^0 ) + \bar{L}^2 \sum_{k=0}^T t_k^2} { \lambda ( (1/p) - \lambda\bar{L}) \sum_{k=0}^T t_k},
\] %
where the expectation is taken over all random choices made by the algorithm. In particular, if the step sizes $\{t_k\}_k$ satisfy $\sum_{k=0}^\infty t_k=\infty$ and $\sum_{k=0}^\infty t_k^2< \infty$, then $\mathbb{E}\left[ \Theta_\lambda(\bar{\bm x})^2 \right] \rightarrow 0$. Moreover, if we take $t_k = \tfrac{1}{\sqrt{T+1}}$ for $k=0,1,\ldots,T$, then the number of iterations needed by StManPPA to obtain a point $\bar{\bm x}\in\M$ satisfying $\mathbb{E}[\Theta_{\lambda}(\bar{\bm x})] \le \epsilon$ is $\mathcal{O}(\epsilon^{-4})$.
\end{theorem}

The proof of Theorem~\ref{prop:stochastic subgradient global rate} can be found in Appendix~\ref{app:stmanppa}. {Again, it makes crucial use of our newly established Riemannian subgradient inequality (see Appendix~\ref{app:conv-anal-local}, Proposition \ref{weakly-inequality}), which allows StManPPA to be analyzed in a similar way as various Euclidean stochastic methods.} We remark that the iteration complexity bound $\mathcal{O}(\epsilon^{-4})$ of StManPPA established in Theorem~\ref{prop:stochastic subgradient global rate} is comparable to that of RSGM established in~\cite{li2019nonsmooth}. 

\section{Extension to Stiefel Manifold Constraint}\label{sec:sequent}
In this section, we consider the matrix analog~\eqref{DPCP-matrix} of problem~\eqref{DPCP}, which also arises in many applications such as certain ``one-shot'' formulations of the ODL and RSR problems (see Section~\ref{subsec:contrib}). Currently, there are two existing approaches for solving problem~\eqref{DPCP-matrix}, namely a sequential linear programming (SLP) approach and an iteratively reweighted least squares (IRLS) approach~\cite{Tsakiris-Vidal-2018}. In the SLP approach, the columns of $\bm{X}$ are extracted one at a time. Suppose that we have already obtained the first $\ell$ columns of $\bm{X}$ ($\ell=0,1,\ldots,q-1$) and arrange them in the matrix $\bm{X}_\ell \in \R^{n\times\ell}$ (with $\bm{X}_0=\bm{0}$). Then, the $(\ell+1)$-st column of $\bm{X}$ is obtained by solving 
\[ 
\min_{\bm{x}\in\R^n} \ \|\bm{Y}^\top\bm{x}\|_1\quad \st\quad\|\bm{x}\|_2=1, \ \bm{X}_\ell^\top\bm{x} = \bm{0}.
\]
This is achieved by the alternating linearization and projection (ALP) method, which generates the iterates via
\begin{align*}
\bm{z}^k &= \argmin_{\bm{z} \in \R^n} \ \|\bm{Y}^\top\bm{z}\|_1 \,\,\, \st \,\,\, \bm{z}^\top \bm{x}^{k-1} = 1,\, \bm{X}_\ell^\top\bm{z} = \bm{0}, \\ 
\bm{x}^k &= \bm{z}^k/\|\bm{z}^k\|_2.
\end{align*}
Note that the $\bm{z}$-subproblem is a linear program, which can be efficiently solved by off-the-shelf solvers. 

In the IRLS approach, the following variant of problem \eqref{DPCP-matrix}, which favors row-wise sparsity of $\bm{Y}^\top\bm{X}$ and has also been studied by Lerman\etal in \cite{Lerman-2015}, is considered:
\be\label{IRLS-prob}
\min_{\bm{X}\in\R^{n\times q}} \ \|\bm{Y}^\top \bm{X}\|_{1,2} \quad\st\quad \bm{X}^\top\bm{X} = \bm{I}_{q}.
\ee
Here, $\|\bm{Y}^\top \bm{X}\|_{1,2}$ denotes the sum of the Euclidean norms of the rows of $\bm{Y}^\top\bm{X}$. The IRLS method for solving~\eqref{IRLS-prob} generates the iterates via
\[ 
\bm{X}^k = \argmin_{\bm{X}\in\R^{n\times q}} \ \sum_{j=1}^p w_{j,k}\|\bm{X}^\top \bm{y}_j\|_2^2\quad \st\quad \bm{X}^\top\bm{X} = \bm{I}_q, 
\] 
where $w_{j,k} = 1/\max\{\delta,\|\bm{X}^{k-1 \top}\bm{y}_j\|_2\}$ and $\delta>0$ is a perturbation parameter to prevent the denominator from being 0. The solution $\bm{X}^k$ can be obtained via an SVD and is thus easy to implement. However, there has been no convergence guarantee for the IRLS method so far.

Recently, Wang\etal\cite{WLS19} proposed a proximal alternating maximization method for solving a maximization version of~\eqref{DPCP-matrix}, which arises in the so-called \emph{$\ell_1$-PCA} problem (see~\cite{Lerman-PIEEE-survey}). However, the method cannot be easily adapted to solve problem~\eqref{DPCP-matrix}.

The similarity between problems~\eqref{DPCP} and~\eqref{DPCP-matrix} suggests that the latter can also be solved by ManPPA, which is indeed the case. However, the SSN method for solving the resulting nonsmooth equation (i.e., the matrix analog of \eqref{nonsmooth-equation}) can be slow, as the dimension of the linear system \eqref{ssn} is high. Here, we propose an alternative method called \emph{sequential ManPPA}, which solves \eqref{DPCP-matrix} in a column-by-column manner and constitutes the fourth main contribution of this paper. The method is based on the observation that the objective function in \eqref{DPCP-matrix} is separable in the columns of $\bm{X}=[\bm{x}_1,\bm{x}_2,\ldots,\bm{x}_n]$. To find $\bm{x}_1$, we simply solve
\[
\min_{\bm{x}_1 \in \R^n} \ \|\bm{Y}^\top \bm{x}_1\|_1 \quad \st \quad \|\bm{x}_1\|_2=1
\]
using ManPPA as it is an instance of problem~\eqref{DPCP}. Now, suppose that we have found the first $\ell$ columns of $\bm{X}$ ($\ell=0,1,\ldots,q-1$) and arrange them in the matrix $\bm{Q}_\ell \in \R^{n\times\ell}$ (with $\bm{Q}_0=\bm{0}$). Then, we find the $(\ell+1)$-st column $\bm{x}_{\ell+1}$ by solving
\be\label{general-vector}
\min_{\bm{x}_{\ell+1}\in\R^n} \|\bm{Y}^\top\bm{x}_{\ell+1}\|_1 \,\,\, \st \,\,\, \|\bm{x}_{\ell+1}\|_2 = 1, \ \bm{Q}_{\ell}^\top \bm{x}_{\ell+1} = \bm{0}.
\ee
As it turns out, problem \eqref{general-vector} is equivalent to the deflation process discussed in \cite{Sun-CDR-part2-2017} for sequentially recovering the columns of an orthogonal dictionary. Specifically, let $\bm{V}_\ell$ be an orthonormal basis of the orthogonal complement of $\bm{Q}_{\ell}$. We can then find $\bm{x}_{\ell+1}$ by solving
	\be\label{prob:deflation}
	\min_{\bm{q} \in \R^{n-\ell}}\| \bm{Y}^\top \bm{V}_\ell \bm{q}\|_1 \quad \st \quad \| \bm{q}\|_2 = 1.
	\ee
Note that \eqref{prob:deflation} has the same form as \eqref{DPCP} and thus can be solved by RSGM or PSGM. By contrast, problem~\eqref{general-vector} has an extra linear constraint and cannot be solved by RSGM or PSGM directly. Nevertheless, our ManPPA can solve both \eqref{general-vector} and \eqref{prob:deflation}. Let us now briefly discuss how to use ManPPA to solve the former. To simplify notation, let us drop the index $\ell$ and denote $\bm{x}=\bm{x}_{\ell+1}$, $\bm{Q}=\bm{Q}_{\ell}$. In the $k$-th iteration of ManPPA, we update the iterate by \eqref{ManPPA-dpcp-retraction}, where the search direction $\bm{d}^k$ is computed by

\[ 
\begin{array}{cl}
\displaystyle\min_{\bm{d}\in\R^n, \atop \bm{u}\in\R^p} & \displaystyle\frac{1}{2}\|\bm{d}\|_2^2+t\|\bm{u}\|_1 \\
\st & \bm{Y}^\top(\bm{x}^k+\bm{d}) = \bm{u},\, \bm{d}^\top [\bm{Q}, \bm{x}^k] = \bm{0}.
\end{array}
\] 
This subproblem can be solved using the SSN-based inexact ALM framework in Section \ref{sec:ManPPA}. We omit the details for succinctness. The sequential ManPPA is guaranteed to find a feasible solution to problem~\eqref{DPCP-matrix}. Moreover, as we shall see in Section~\ref{sec:numerical}, it is computationally much more efficient than the ALP and IRLS methods on the ODL problem and the DPCP formulation of the RSR problem. However, it remains open whether the solution found by sequential ManPPA is a stationary point of problem~\eqref{DPCP-matrix}. We leave this question for future research.

\section{Numerical Experiments}\label{sec:numerical}

In this section, we compare our proposed ManPPA, StManPPA, and sequential ManPPA with the existing methods PSGM-MBLS, ALP, and IRLS. We do not include RSGM with diminishing step sizes in our comparison, as the numerical results in \cite{zhu2018dualpcp} show that they are slower than PSGM-MBLS and cannot achieve high accuracy. In all the tests, we used the step size $t=0.1$ for ManPPA and set the maximum number of iterations of ManPPA, inexact ALM, and SSN to 100, 30, and 20, respectively. We stopped ManPPA if the relative change in the objective values satisfies $|f(\bm{x}^k) - f(\bm{x}^{k-1})| / f(\bm{x}^{k-1}) \leq 10^{-9}$. In the $k$-th iteration of ManPPA, we stopped the inexact ALM if the primal and dual feasibility of problem \eqref{ManPPA-dpcp-sub-rewrite} satisfy
\[
\max\Big\{ \sqrt{ \| \bm{Y}^\top \bm{d}^{j+1}+\bm{c} - \bm{u}^{j+1} \|_2^2 + ((\bm{d}^{j+1})^\top \bm{x})^2 }, \| \nabla \psi_{j}(\bm{d}^{j+1}) \|_2 \Big\} \leq \epsilon_k = 0.1^k.
\]
For SSN, we used the same termination criteria as the ones given in \cite{Sun-lasso-2018} and solved the linear equation \eqref{eq:ssn-a} by Cholesky decomposition. We refer the reader to the companion technical report~\cite{chen2020manifold} for the setting of the parameters in the inexact ALM and SSN.

For StManPPA, we used the piecewise geometrically diminishing step sizes $t_k=\beta^{\lfloor k/p \rfloor}t_0$ for $k=0,1,\ldots,pT$ with $t_0=0.6$ and $T=500$. Such step sizes are motivated by those used in PSGM~\cite{zhu2018dualpcp}. We use StManPPA-$\beta$ to specify the parameter $\beta$ used in the algorithm. We stopped the algorithm if the relative change in the objective values satisfies $|f(\bm{x}^k) - f(\bm{x}^{k-1})| / f(\bm{x}^{k-1}) \leq 10^{-12}$.
For PSGM-MBLS, ALP, and IRLS, we used their default settings of the parameters. We stopped ALP if the change in the objective values satisfies
$|f(\bm{x}^k) - f(\bm{x}^{k-1})| \leq 10^{-6}$, while we stopped IRLS if the change in the objective values satisfies $|f(\bm{x}^k) - f(\bm{x}^{k-1})| \leq 10^{-11}$. With these stopping criteria, the solutions returned by these algorithms usually achieve the same accuracy.

\subsection{DPCP Formulations of the RSR Problem}

In this section, we consider the DPCP formulations of the RSR problem. For the recovery of a vector that is orthogonal to the inlier subspace (the \emph{vector} case), we compared the performance of ManPPA, StManPPA, PSGM-MBLS, ALP, and IRLS on problem~\eqref{DPCP}. For the recovery of the entire inlier subspace of known dimension $d$ (the \emph{matrix} case), we compared the performance of sequential ManPPA, PSGM-MBLS, ALP, and IRLS on problem~\eqref{DPCP-matrix} with $q=n-d$.

\subsubsection{Synthetic Data}

We first test the algorithms on synthetic data. The data matrix takes the form $\bm{Y} = [\bm{X}, \bm{O}]\in\br^{n\times (p_1+p_2)}$. We generated the inlier data $\bm{X}$ as $\bm{X} = \bm{Q}\bm{C}$, where $\bm{Q}\in\R^{n\times d}$ is an orthonormal matrix and $\bm{C}\in\R^{d\times p_1}$ is a coefficient matrix. The matrix $\bm{Q}$ was generated by orthonormalizing an $n\times d$ standard Gaussian random matrix via QR decomposition, while the matrix $\bm{C}$ was generated as a $d\times p_1$ standard Gaussian random matrix. The outlier data $\bm{O} \in\R^{n\times p_2}$ was generated as a standard Gaussian random matrix. Finally, the columns of the matrix $\bm{Y}$ were normalized. The $d$-dimensional subspace spanned by $\bm{X}$ is denoted by $\CS$ and its orthogonal complement is denoted by $\CS^\perp$.

{\bf Vector case.} We set the initial point $\bm{x}^0$ of all algorithms to be the eigenvector of $\bm{Y}\bm{Y}^\top$ corresponding to the eigenvalue with minimum magnitude. We compared the performance of the algorithms on problem~\eqref{DPCP} with different dimension $n$, number of inliers $p_1$, and number of outliers $p_2$ in Figure \ref{fig:DPCP-ManPPA-LP-IRLS-Sub}. The first row of Figure \ref{fig:DPCP-ManPPA-LP-IRLS-Sub} reports the principal angle\footnote{The principal angle is the distance between $\bm{x}^k$ and $\CS^\perp$. Any optimal solution $\bm{x}^*$ to problem \eqref{DPCP} is orthogonal to the inlier subspace $\CS$ \cite[Theorem 1]{zhu2018dualpcp}. Using the Lipschitz continuity of  $f$, we know that $\theta$ also measures the function value gap $f(\bm{x})-f(\bm{x}^*)$.}
$\theta$ between $\bm{x}^k$ and $\CS^{\perp}$ versus the iteration number. The second row reports $\theta$ versus CPU time. Note that $\bm{x}^k = \sin(\theta)\bm{n}+ \cos(\theta)\bm{s}$, where $\bm{n}=\Proj_{\CS}(\bm{x}^k)/\normtwo{\Proj_{\CS}(\bm{x}^k)}_2$ and $\bm{s}=\Proj_{\CS^{\perp}}(\bm{x}^k)/\normtwo{\Proj_{\CS^{\perp}}(\bm{x}^k)}_2$. From Figure \ref{fig:DPCP-ManPPA-LP-IRLS-Sub}, we see that PSGM-MBLS is the fastest, while ManPPA is slightly slower. However, they are both much faster than other compared methods. Moreover, the principal angle $\theta$ of ManPPA decreases much faster than PSGM-MBLS in terms of iteration number. This can be attributed to the quadratic convergence rate of ManPPA (Theorem \ref{thm:local_rate}).

\begin{figure}[htb]
	\begin{center}
		\subfigure[$n=30,p_1=500,p_2=1167$]
{\includegraphics[width=0.4\linewidth]{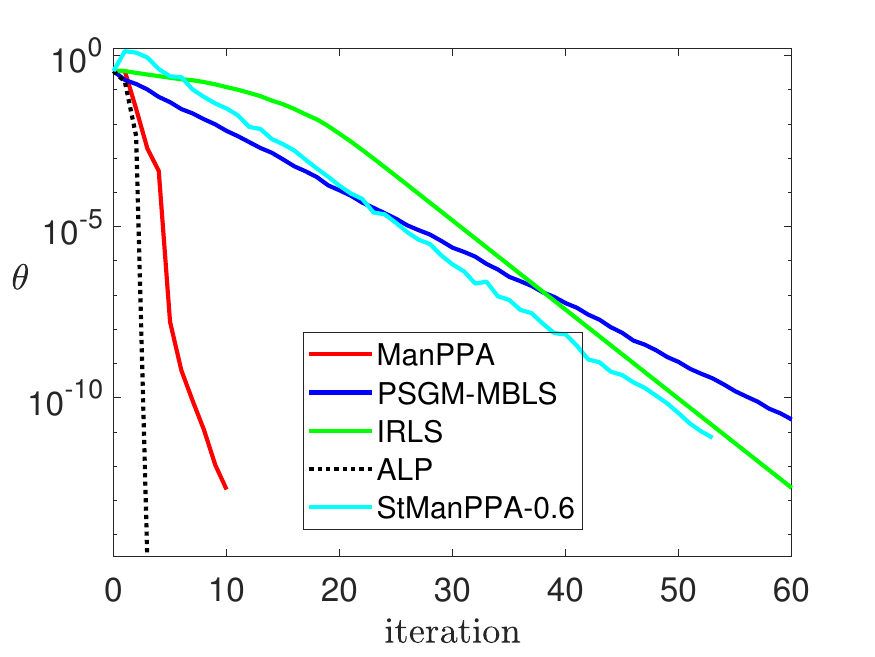}}
        		\subfigure[$n=30,p_1=500,p_2=1167$]{\includegraphics[width=0.4\linewidth]{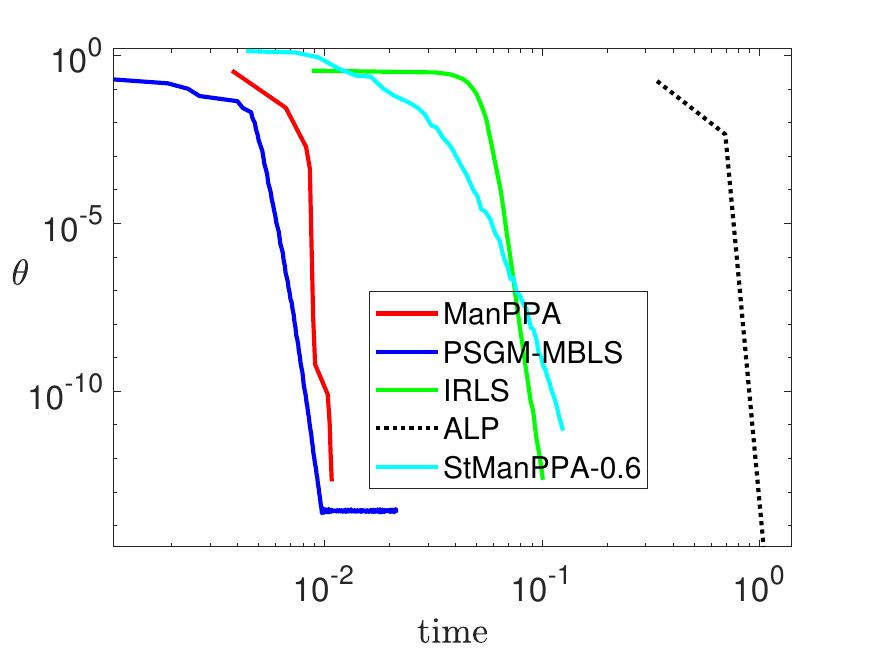}}
		\caption{Numerical results for the DPCP formulation~\eqref{DPCP}. (a): Principal angle versus iteration number. (b): Principal angle versus CPU time.}
		\label{fig:DPCP-ManPPA-LP-IRLS-Sub}
	\end{center}
\end{figure}

In Figure \ref{fig:DPCP-fit-curve} we report the quadratic fitting curves of the different algorithms. As shown in \cite{zhu2018dualpcp}, the DPCP formulation can tolerate $O((\# \text{inliers})^2)$ outliers; i.e., $p_2 =\mathcal{O}(p_1^2)$. For different $p_2\in\{40,80,120,\ldots,600\}$, we find the smallest $p_1\in\{60,70,80,\ldots,260\}$ such that $\theta<10^{-1}$. Here, the principal angle $\theta$ is the mean value of $10$ trials; i.e., we find pairs $(p_1,p_2)$ such that for a fixed $p_1$, $p_2$ is the largest number of outliers that can be tolerated. We then use a quadratic function to fit these pairs $(p_1,p_2)$. A higher curve indicates that more outliers can be tolerated and hence the algorithm is more robust. From Figure \ref{fig:DPCP-fit-curve}, we see that the curve corresponding to PSGM-MBLS is the lowest one and thus the least robust, while ManPPA and StManPPA are more robust. In Figure \ref{fig:DPCP-cpu} we report the CPU time versus $p_1$ and $p_2$. {For each algorithm, the shadow area represents the standard deviation (std) of 10 random trials, while the line within the shadow is the mean of those trials. From the left two subfigures of Figure \ref{fig:DPCP-cpu}, we see that the stds of IRLS and ALP are quite significant. In particular, they are usually more than ten times larger than the stds of other compared algorithms. To better illustrate the stds of the other four algorithms, we plot their CPU times in the right two subfigures of Figure \ref{fig:DPCP-cpu}. From these two figures, we see that ManPPA has a larger std than those of the other three algorithms.} Overall, we see that PSGM-MBLS is the fastest and ManPPA is second, and they are both much faster than the other compared algorithms. Figures \ref{fig:DPCP-fit-curve} and \ref{fig:DPCP-cpu} suggest that ManPPA is slightly slower than PSGM-MBLS but is more robust. Moreover, the choice of the parameter $\beta$ for StManPPA is crucial and challenging, as StManPPA-0.9 is faster but less robust than StManPPA-0.8. We leave the determination of the best parameter $\beta$ for StManPPA as a future work.

\begin{figure}[htb]
 	\begin{center}
 		\includegraphics[width=0.5\linewidth]{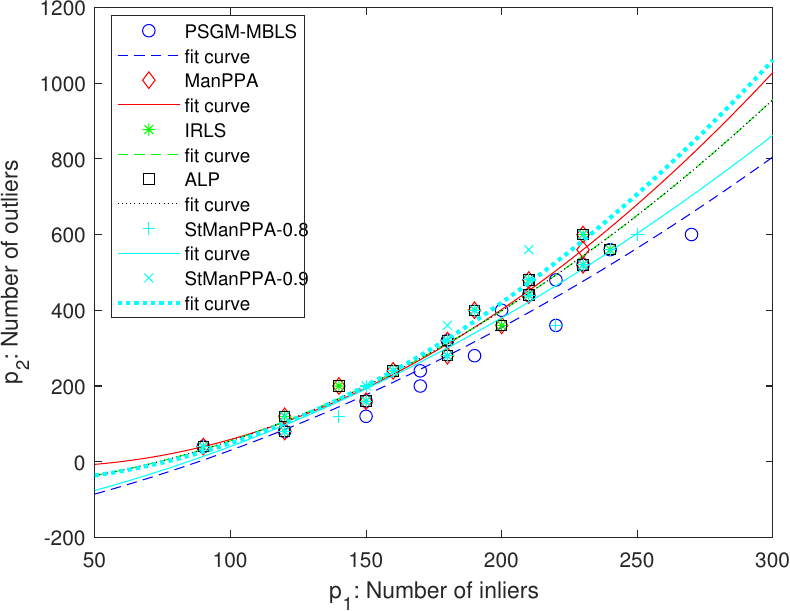}
 		\caption{Quadratic fitting curves ($n=30$).}
 		\label{fig:DPCP-fit-curve}
 	\end{center}
\end{figure}

\begin{figure}[htb]
 	\begin{center}
 		\includegraphics[width=0.38\linewidth]{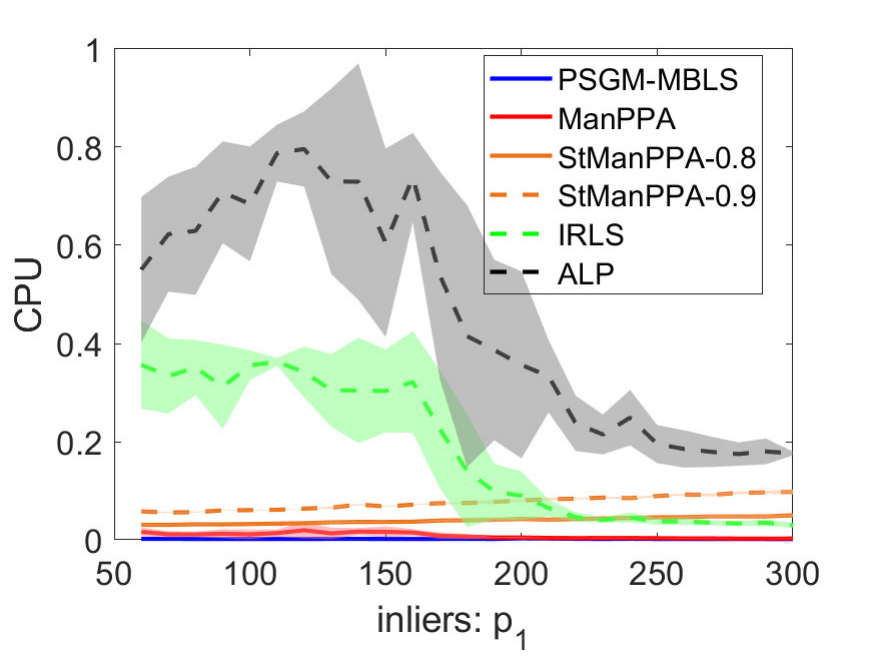}
 		 		\includegraphics[width=0.38\linewidth]{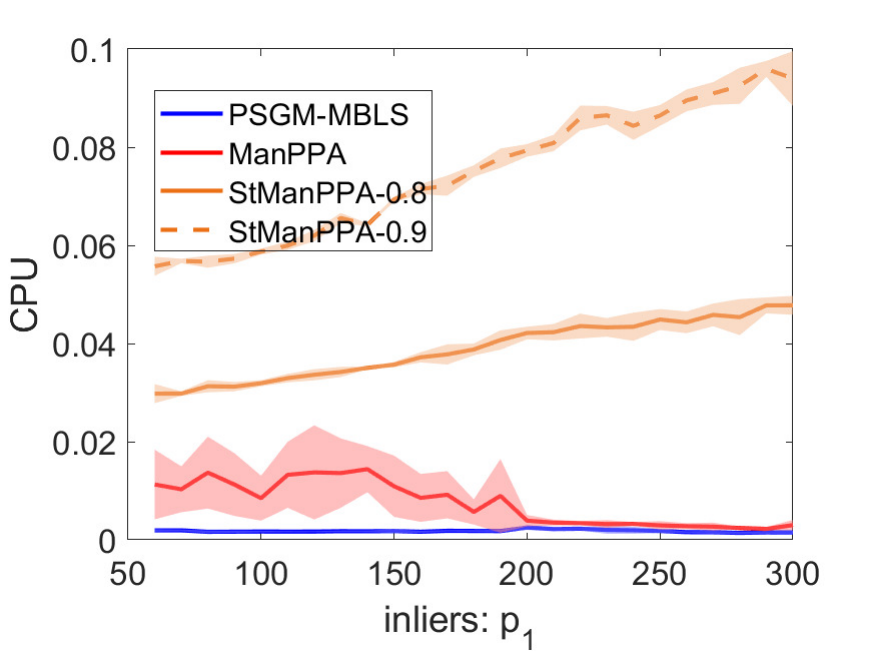}
 		\includegraphics[width=0.38\linewidth]{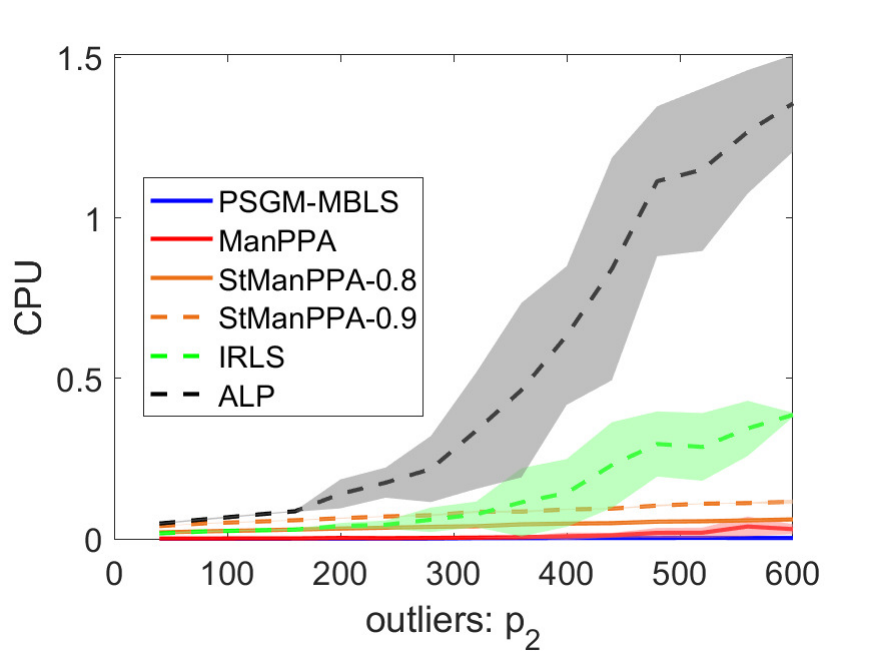}
 		 		 		\includegraphics[width=0.38\linewidth]{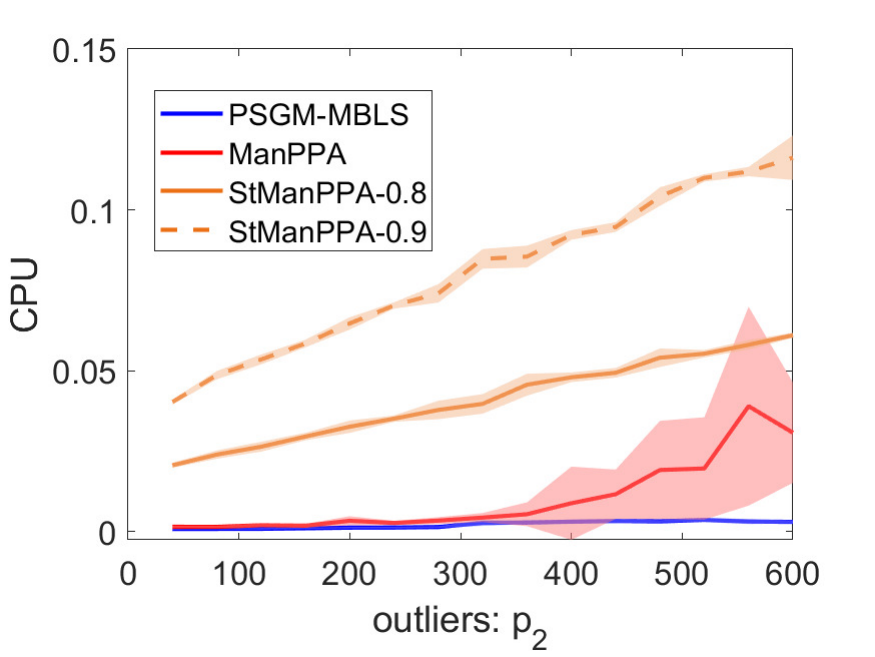}
 		\caption{Upper: CPU time versus the number of inliers $p_1$ ($n=30$, $p_2=320$). Lower: CPU time versus the number of outliers $p_2$ ($n=30$, $p_1=200$). {The shadow area corresponds to the std and the line within the shadow is the mean of 10 random trials.}}
 		\label{fig:DPCP-cpu}
 	\end{center}
\end{figure}

\begin{figure}[!htb]
	\begin{center}
		\includegraphics[width=0.42\linewidth]{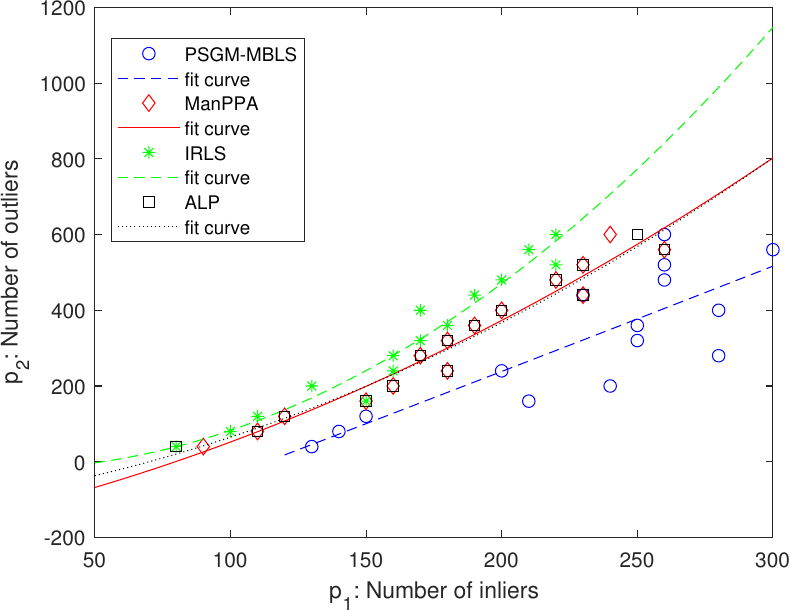} \\
		\includegraphics[width=0.38\linewidth]{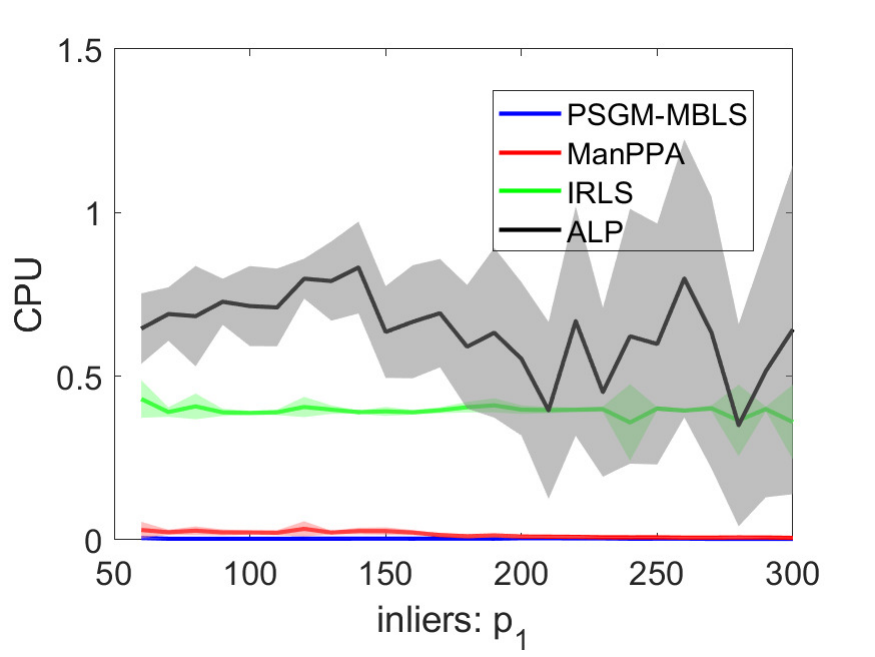}
				\includegraphics[width=0.38\linewidth]{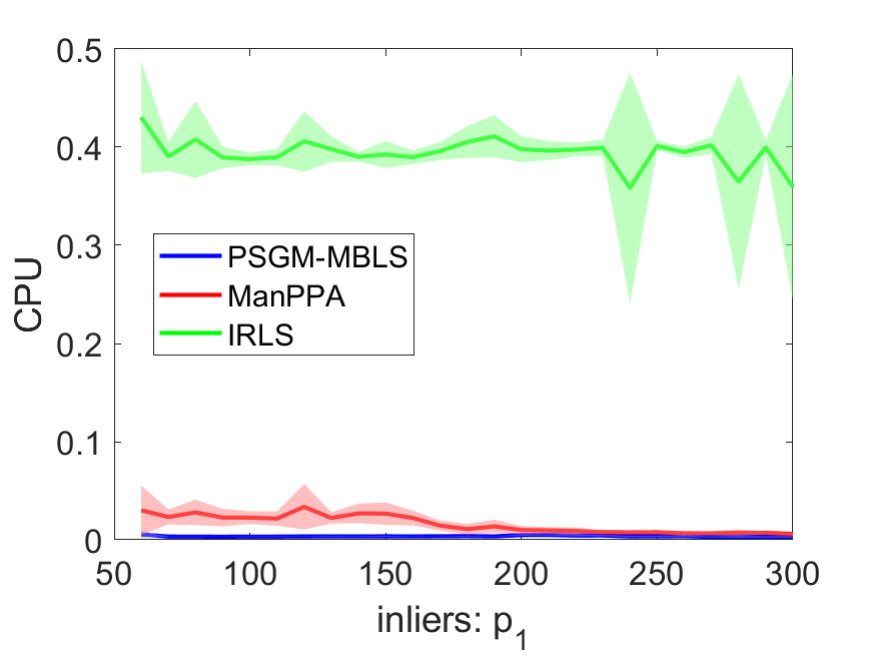}
		\includegraphics[width=0.38\linewidth]{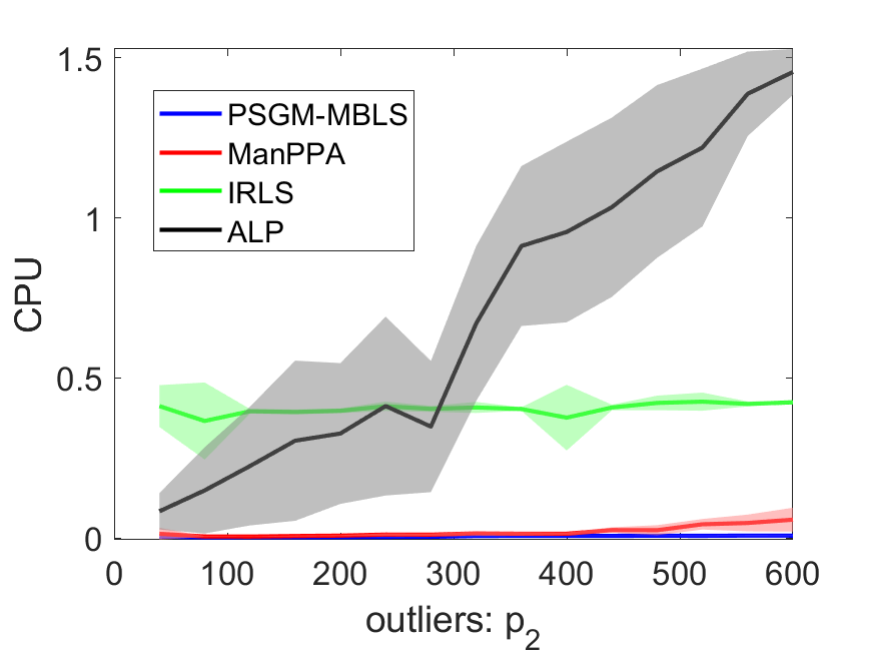}
				\includegraphics[width=0.38\linewidth]{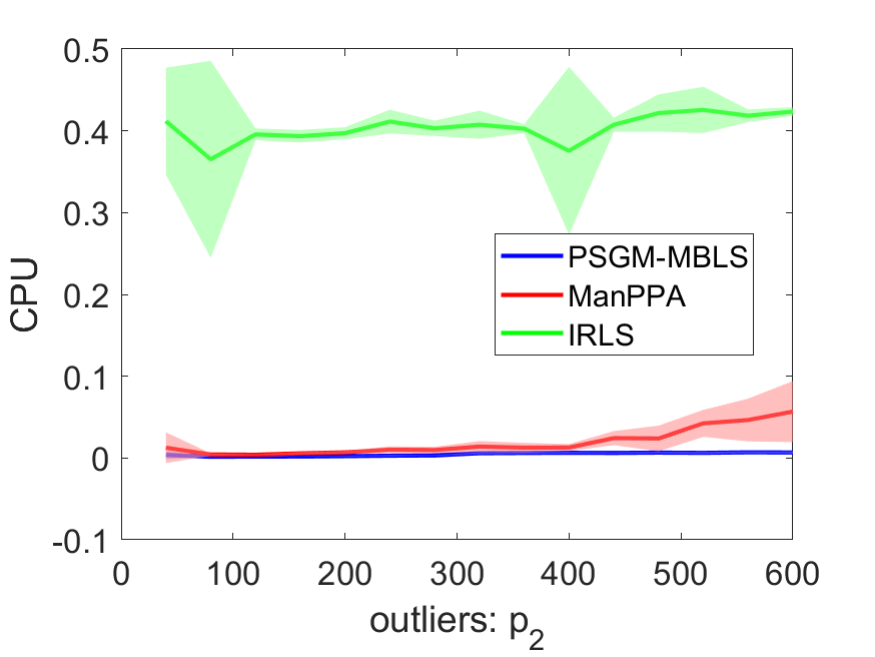}
	    \caption{Comparison on the DPCP formulation \eqref{DPCP-matrix} with $n=30$, $q=2$. First row: Quadratic fitting curves. Second row: CPU time versus the number of inliers $p_1$ ($p_2=320$). Third row: CPU time versus the number of outliers $p_2$ ($p_1=200$). {The shadow area corresponds to the std and the line within the shadow is the mean of 10 random trials.}}
	    \label{fig:DPCP-fit curve-mult-col1}
\end{center}
\end{figure}

\begin{figure}[!htb]
	\begin{center}
	    \includegraphics[width=0.42\linewidth]{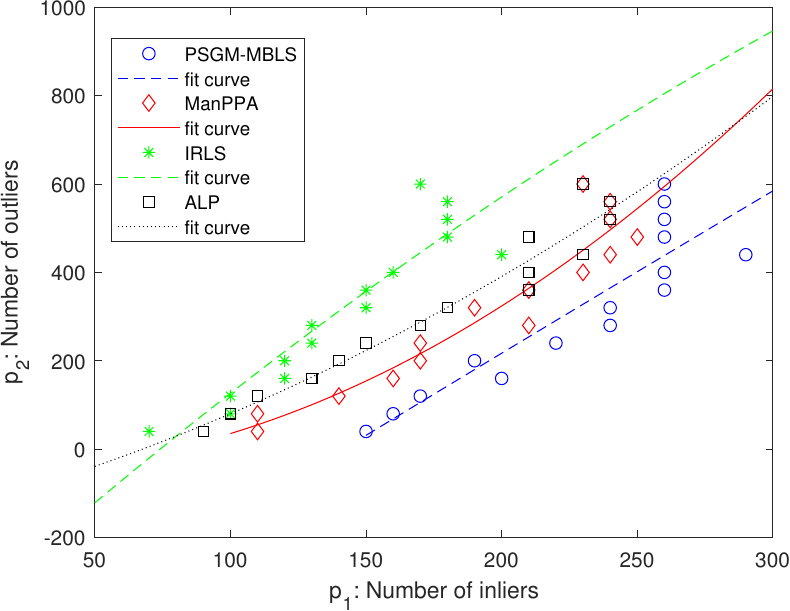} \\
	    		\includegraphics[width=0.38\linewidth]{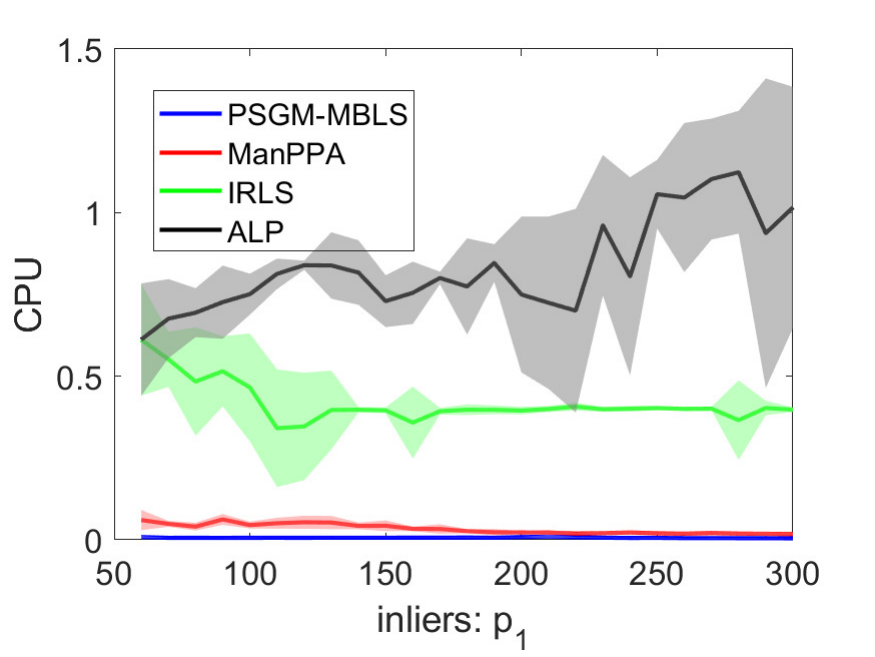}
		\includegraphics[width=0.38\linewidth]{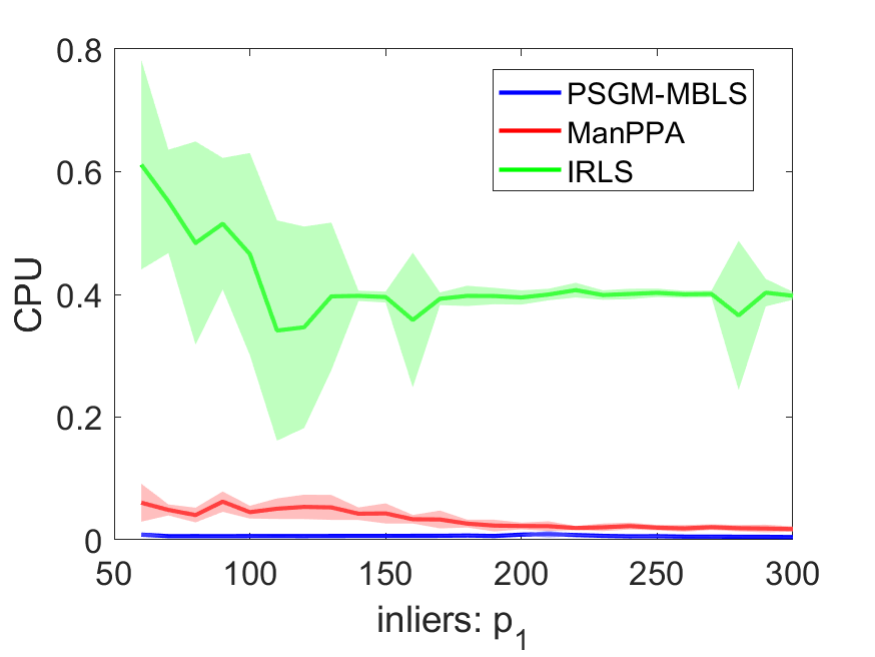}
				\includegraphics[width=0.38\linewidth]{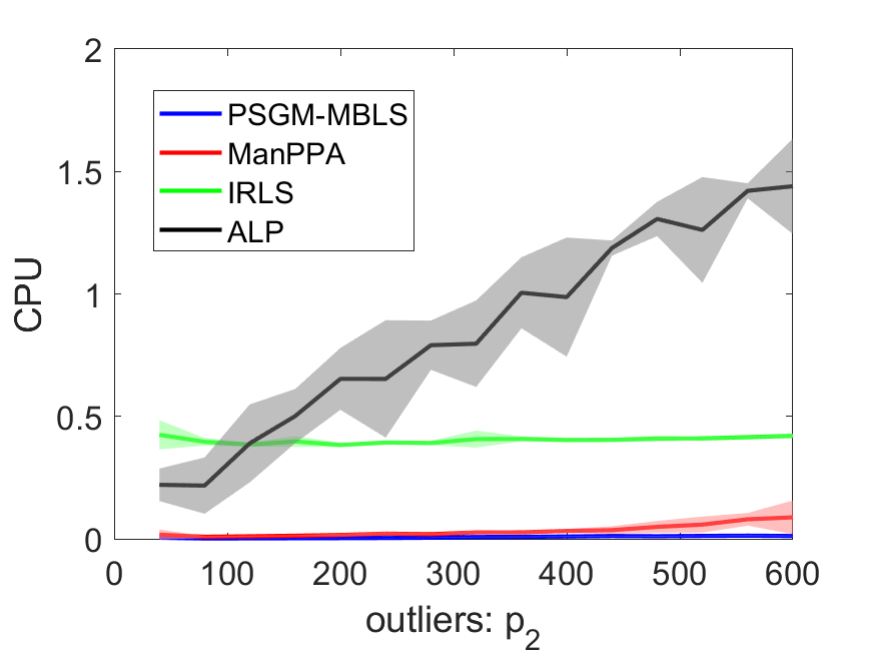}
		\includegraphics[width=0.38\linewidth]{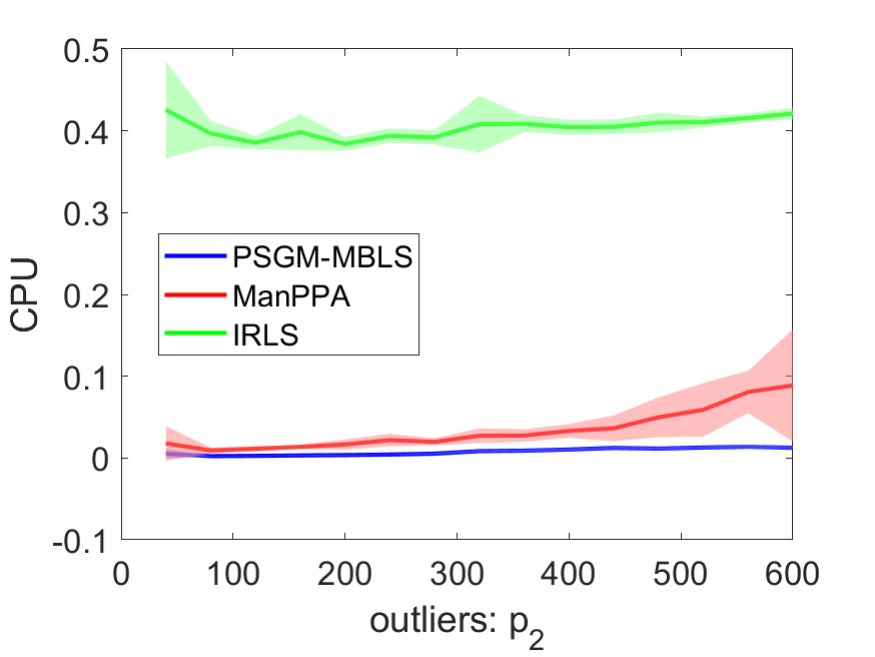}
		\caption{Comparison on the DPCP formulation \eqref{DPCP-matrix} with $n=30$, $q=4$. First row: Quadratic fitting curves. Second row: CPU time versus the number of inliers $p_1$ ($p_2=320$). Third row: CPU time versus the number of outliers $p_2$ ($p_1=200$). {The shadow area corresponds to the std and the line within the shadow is the mean of 10 random trials.}}
		\label{fig:DPCP-fit curve-mult-col2}
	\end{center}
\end{figure}

{\bf Matrix case.} We solved problem \eqref{DPCP-matrix} using sequential ManPPA and compared its performance with {PSGM-MBLS (applied to \eqref{prob:deflation})}, ALP, and IRLS. We report the results for $q=2$ and $q=4$ in Figures \ref{fig:DPCP-fit curve-mult-col1} and \ref{fig:DPCP-fit curve-mult-col2}, respectively. {Since ALP has a very high std, for better illustration of the CPU time comparison, we exclude it from the right two subfigures of Figures \ref{fig:DPCP-fit curve-mult-col1} and \ref{fig:DPCP-fit curve-mult-col2}.} The results suggest that sequential ManPPA is not as robust as IRLS but is the second most efficient one among the four compared algorithms. {Moreover, we find that although PSGM-MBLS is very efficient, its fitting curve is not very good. This is due to the fact that PSGM-MBLS is very sensitive to the choice of step size.}

\subsubsection{Real 3D Point Cloud Road Data}
Next, we compared ManPPA with PSGM-MBLS on the road detection challenge of the KITTI dataset \cite{Geiger2013kittidata}.  This dataset contains image data together with the corresponding 3D points collected by a rotating 3D laser scanner. Similar to \cite{zhu2018dualpcp}, we only used the $\ang{360}$ 3D point clouds to determine which points lie on the road plane (inliers) and which do not (outliers). By using homogeneous coordinates, this can be cast as a robust hyperplane learning problem~\eqref{DPCP} in $\R^4$. As reported in \cite{zhu2018dualpcp}, PSGM-MBLS is the fastest algorithm when compared with other state-of-the-art methods. Thus, we only compared the performance of ManPPA and PSGM-MBLS on problem~\eqref{DPCP}. Table \ref{tab:DPCP-KITTI} reports the area under the Receiver Operator Curve (ROC) and the CPU time. We see that all ROC values of ManPPA are better than those of PSGM-MBLS, with some sacrifice on the CPU time.

\begin{table*}[ht]\small
	\centering
	\caption{Area under ROC and CPU time for annotated 3D point clouds with index 0, 21 in KITTY-CITY-48 and 1, 45, 120, 137, 153 in KITTYCITY-5. The number in parenthesis is the percentage of outliers.}
	\begin{tabular}{c|cc|ccccc}
		\hline
		\multicolumn{1}{c}{} & \multicolumn{2}{c|}{KITTY-CITY-48} & \multicolumn{5}{c}{KITTY-CITY-5} \bigstrut\\
		\hline
		\multicolumn{1}{c}{} & 0 (56$\%$)    & 21 (57$\%$)   & 1 (37$\%$)  & 45 (38$\%$)    & 120 (53$\%$)  & 137 (48$\%$)  & 153(67$\%$) \bigstrut\\
		\hline
		\multicolumn{8}{c}{Area under ROC} \bigstrut\\
		\hline
		ManPPA & 0.99437  & 0.99077  & 0.99810  & 0.99898  & 0.87629  & 0.99969  & 0.75481  \bigstrut[t]\\
		PSGM-MBLS & 0.99420  & 0.99062  & 0.99802  & 0.99891  & 0.86782  & 0.99968  & 0.74933  \bigstrut[b]\\
		\hline
		\multicolumn{8}{c}{CPU time} \bigstrut\\
		\hline
		ManPPA & 0.129  & 0.174  & 0.091  & 0.066  & 0.106  & 0.108  & 0.066  \bigstrut[t]\\
		PSGM-MBLS & 0.028  & 0.015  & 0.034  & 0.029  & 0.029  & 0.014  & 0.017  \bigstrut[b]\\
		\hline
	\end{tabular}%
	\label{tab:DPCP-KITTI}%
\end{table*}%

\subsection{ODL Problem}
We generated instances of the ODL problem by first randomly generating an orthogonal matrix $\hat{\bm X}\in\R^{n\times n}$ and a Bernoulli-Gaussian matrix $\hat{\bm A}\in \R^{n\times p}$ with parameter $\gamma$ (see, e.g.,~\cite{Spielman-Wang-Wright-2012}), then setting $\bm{Y} = \hat{\bm X}\hat{\bm A}$.

{\bf Vector case.} We first compared the performance of ManPPA and StManPPA with ALP, IRLS, and PSGM-MBLS on problem~\eqref{DPCP} with $n=30$, $p =\lceil 10 n^{1.5}\rceil$. Figures \ref{fig:DL-ManPPA-ALP-IRLS1} and \ref{fig:DL-ManPPA-ALP-IRLS2} report the iteration numbers and CPU times of the compared algorithms. The quantity $\theta$ is the angle between $\bm{x}^k$ returned by the algorithm and its nearest column in $\hat{\bm X}$. From Figures \ref{fig:DL-ManPPA-ALP-IRLS1} and \ref{fig:DL-ManPPA-ALP-IRLS2}, we see that PSGM-MBLS is the fastest algorithm in terms of CPU time, while ManPPA is slightly slower. However, they are both much faster than the other compared algorithms. Moreover, we see that ManPPA is much faster than PSGM-MBLS in terms of iteration number. This again can be attributed to the quadratic convergence rate of ManPPA (Theorem \ref{thm:local_rate}).

\begin{figure}[htb]
	\begin{center}
		\includegraphics[width=0.35\linewidth]{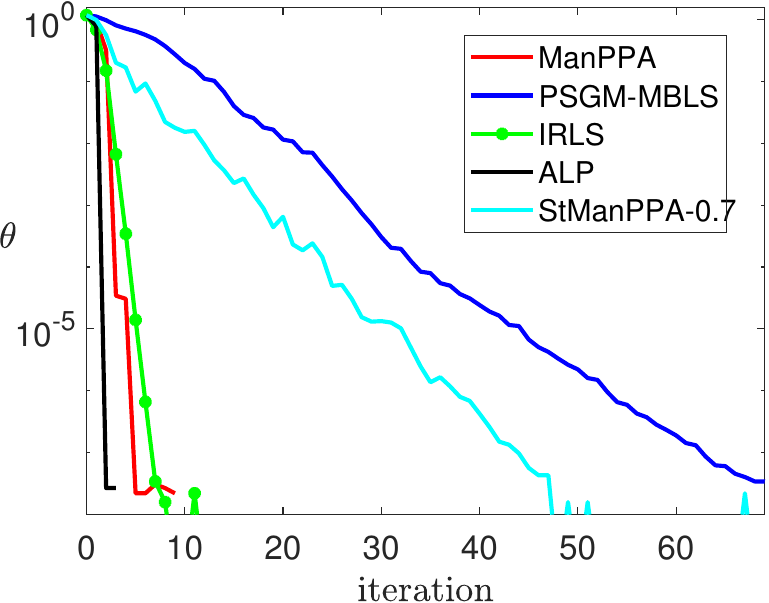}
		\includegraphics[width=0.35\linewidth]{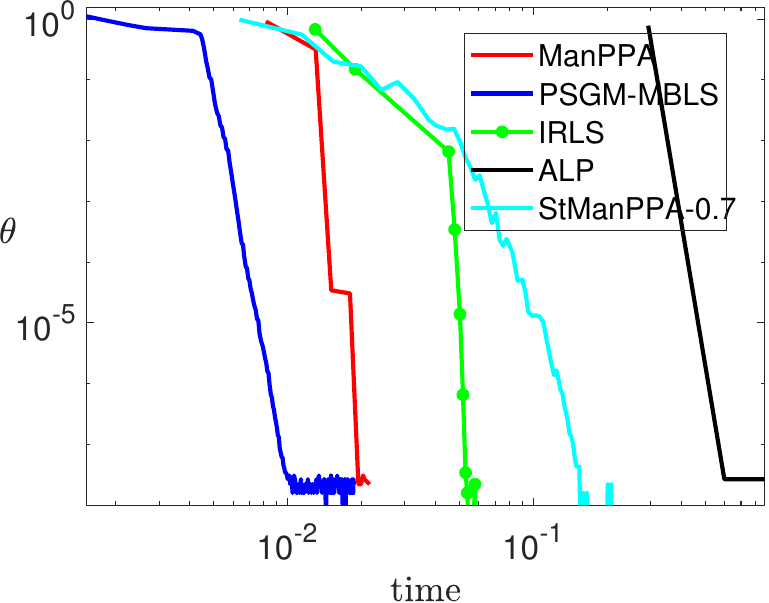}
		\caption{Numerical results for the ODL problem~\eqref{DPCP}: $n=30$, $p= \lceil 10n^{1.5}\rceil$, $\gamma=0.1$. }
		\label{fig:DL-ManPPA-ALP-IRLS1}
\end{center}
\end{figure}

\begin{figure}[htb]
	\begin{center}
		\includegraphics[width=0.35\linewidth]{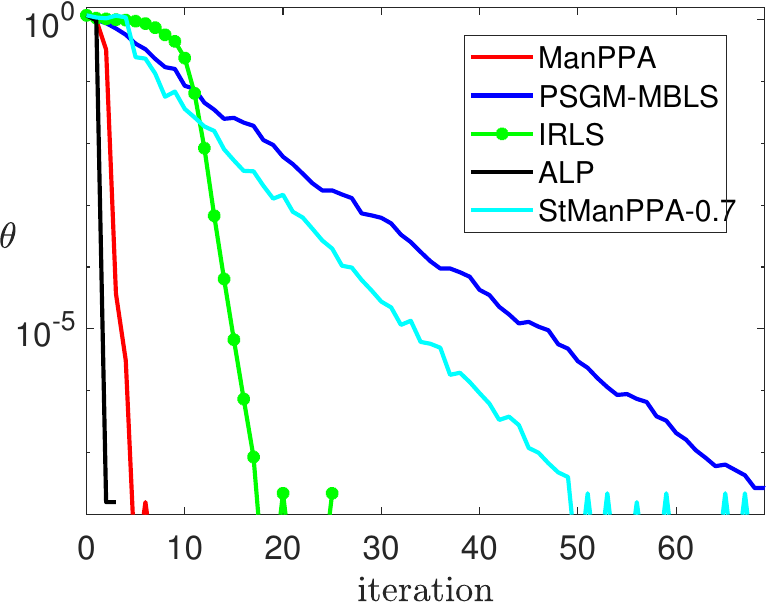}
		\includegraphics[width=0.35\linewidth]{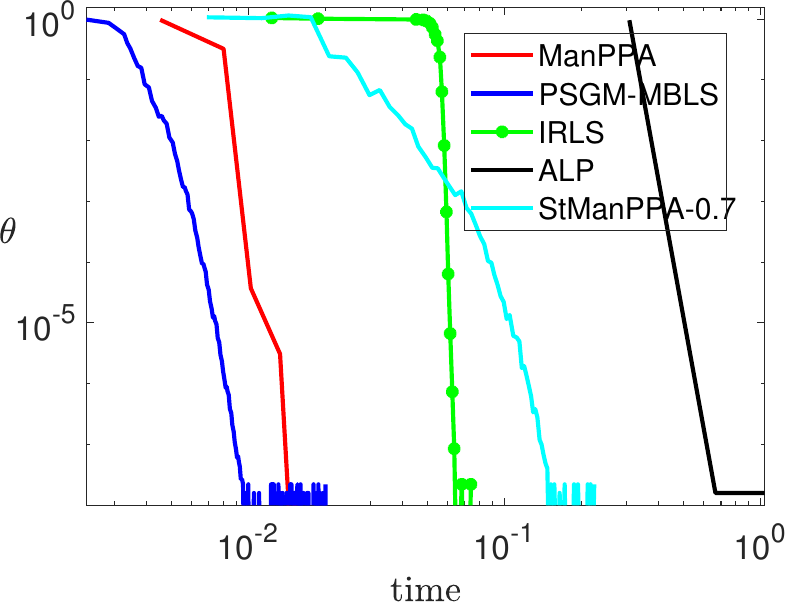}
		\caption{Numerical results for the ODL problem~\eqref{DPCP}: $n=30$, $p= \lceil 10n^{1.5}\rceil$, $\gamma=0.3$. }
		\label{fig:DL-ManPPA-ALP-IRLS2}
	\end{center}
\end{figure}

In Figures \ref{fig:DL-fit curve-single-01} and \ref{fig:DL-fit curve-single-03} we report the linear fitting curves for $\log(n)$ and $\log(p)$ and CPU times of ManPPA, StManPPA-0.8, StManPPA-0.9, IRLS, ALP, and PSGM-MBLS. Note that for the ODL problem, it has been found empirically in  \cite{bai2018subgradient} that the sample size $p$ and dimension $n$ should satisfy $p = O(n^2)$ to guarantee recovery. The linear fitting curves were found in the following manner. For a given dimension $n\in\{5,10,15\ldots,50\}$, we find the smallest sample number $p\in 2n+\{10,20,30,\ldots,800\}$ such that $\theta<10^{-1}$. Here the principal angle $\theta$ is the mean value of $10$ trials. We then use a linear function to fit the points $\{ (\log(n),\log(p)) \}_{n,p}$. From Figures \ref{fig:DL-fit curve-single-01} and \ref{fig:DL-fit curve-single-03}, we find that PSGM-MBLS is the fastest but its fitting curve is high, which suggests that it is not robust. This is because PSGM-MBLS is very sensitive to the choice of step size. ManPPA appears to be the second fastest but is very robust based on the fitting curve. 

\begin{figure}[!htb]
	\begin{center}
		\includegraphics[width=0.42\linewidth]{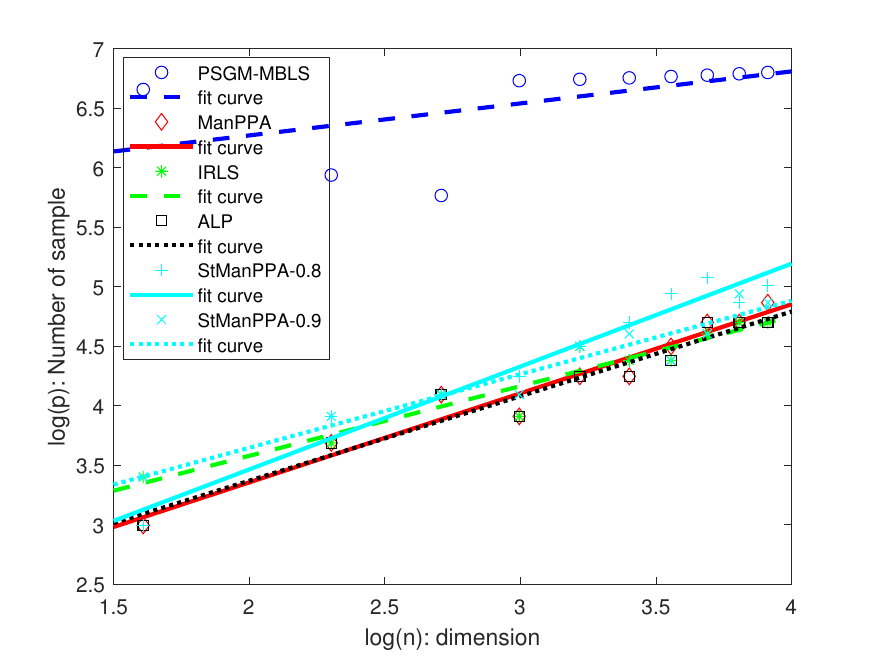} \\
		\includegraphics[width=0.38\linewidth]{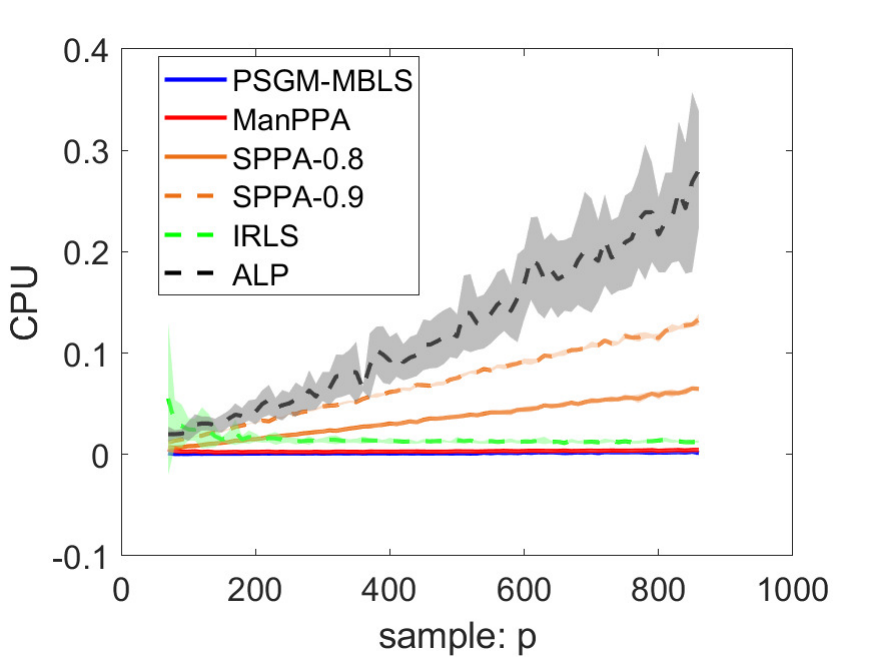}
		\includegraphics[width=0.38\linewidth]{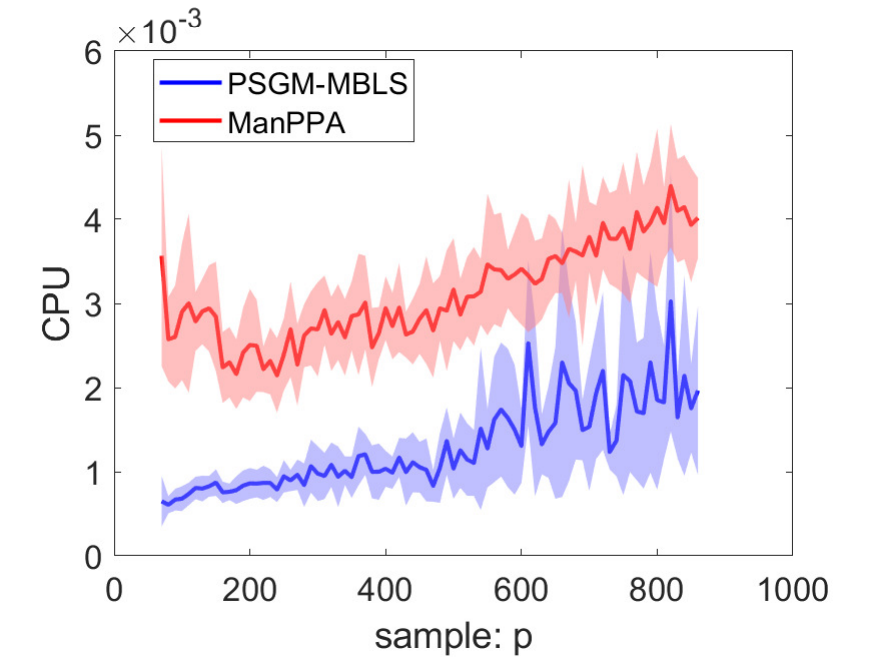}
		\includegraphics[width=0.38\linewidth]{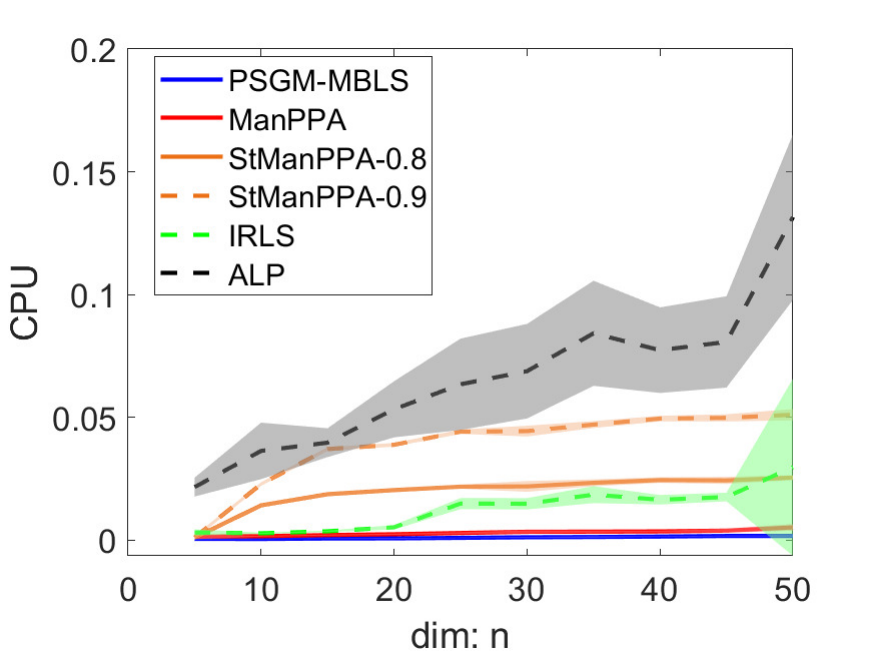}
		\includegraphics[width=0.38\linewidth]{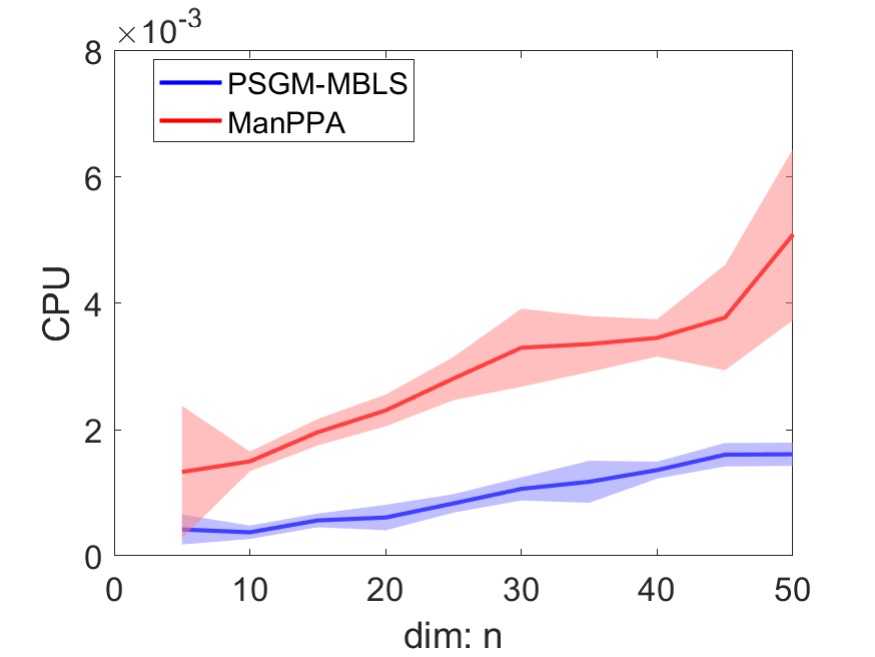}
		\caption{Comparison on the ODL problem~\eqref{DPCP} with $\gamma=0.1$. First row: Fitting curves. Second row: CPU time versus the number of samples $p$, ($n=30$). Third row: CPU time versus the number of dimension $n$, ($p=300$).  {The shadow area corresponds to the std and the line within the shadow is the mean of 10 random trials.}}
		\label{fig:DL-fit curve-single-01}
		\end{center}
\end{figure}

\begin{figure}[!htb]
	\begin{center}
		\includegraphics[width=0.42\linewidth]{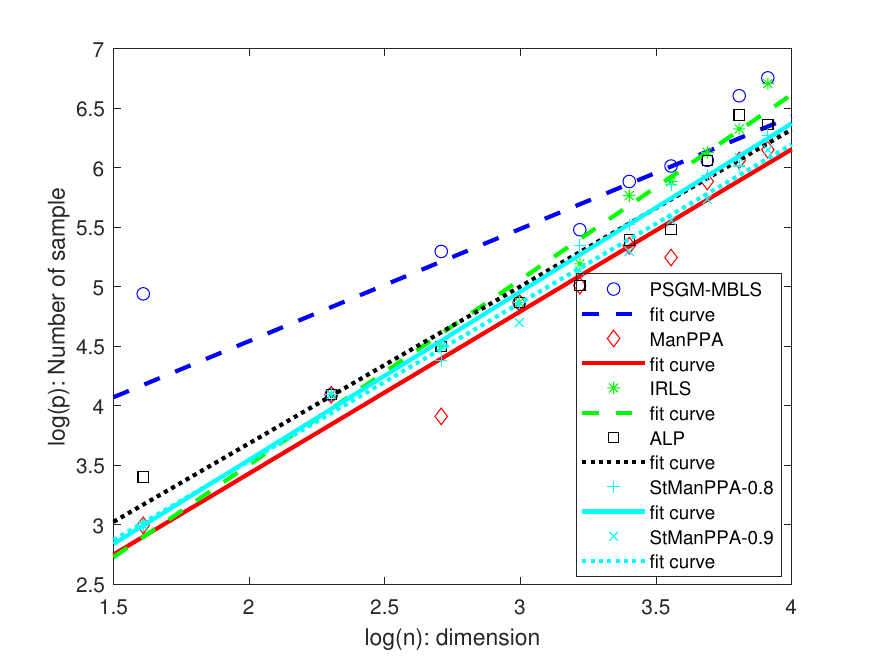} \\
		\includegraphics[width=0.38\linewidth]{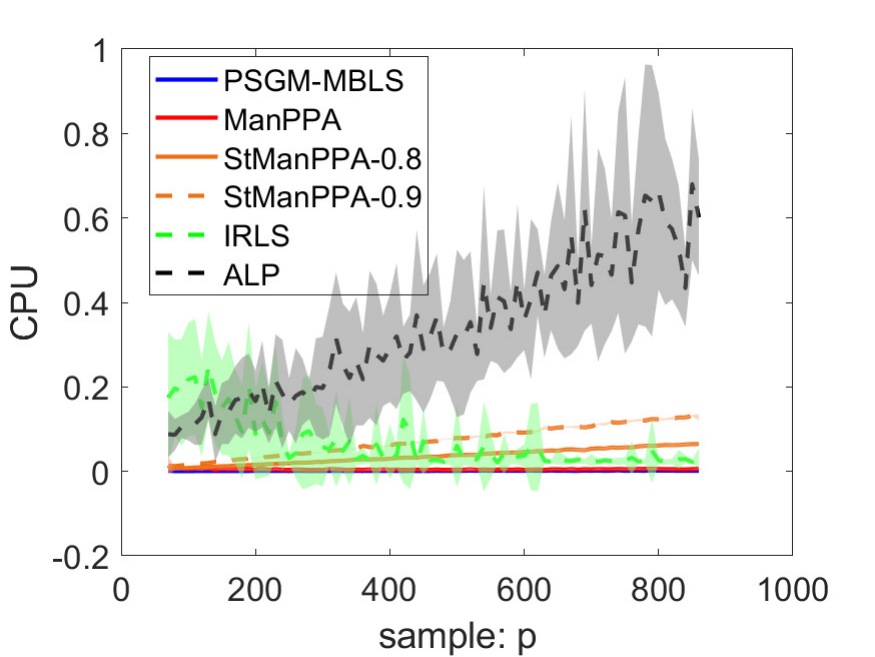}
				\includegraphics[width=0.38\linewidth]{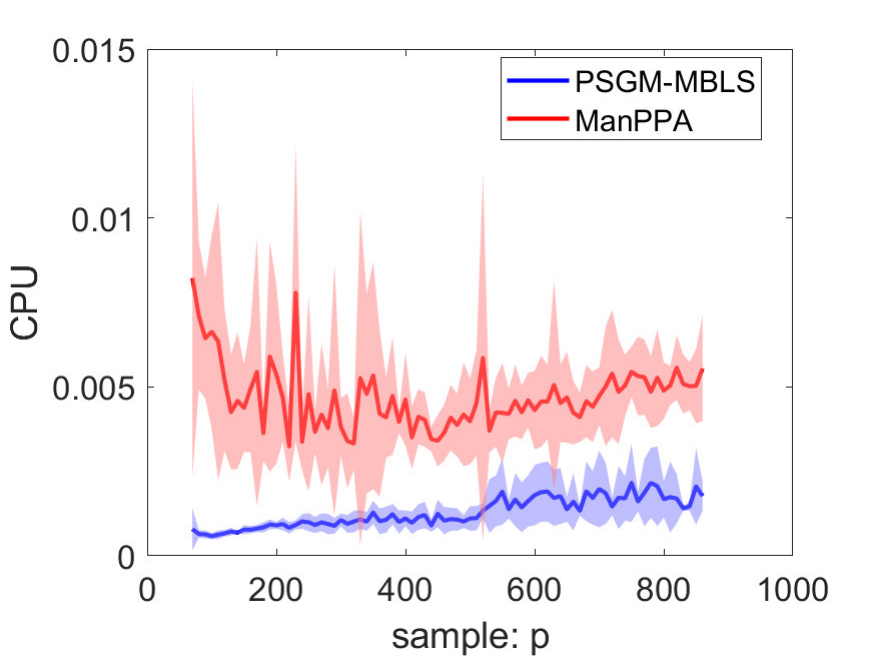}
		\includegraphics[width=0.38\linewidth]{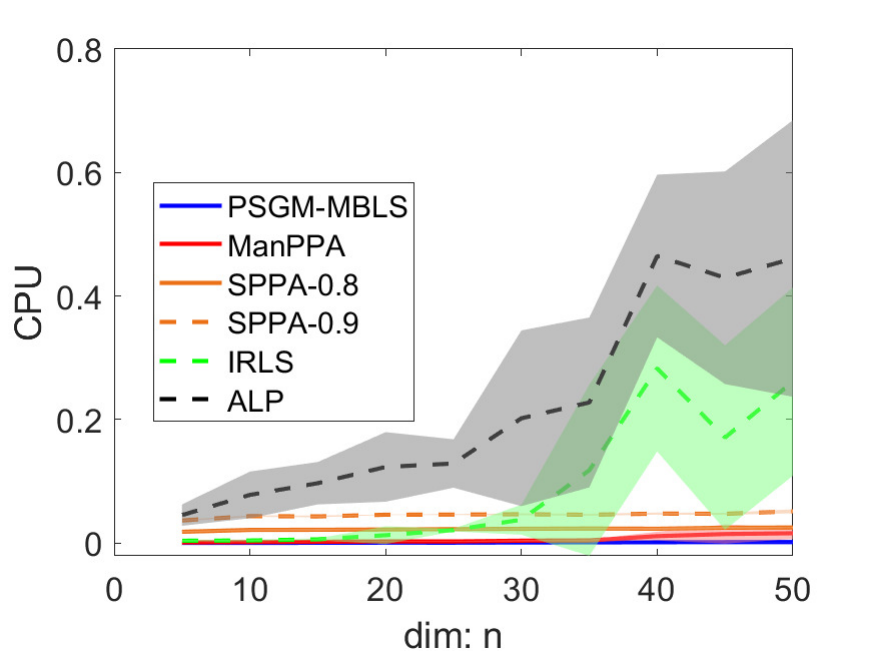}
				\includegraphics[width=0.38\linewidth]{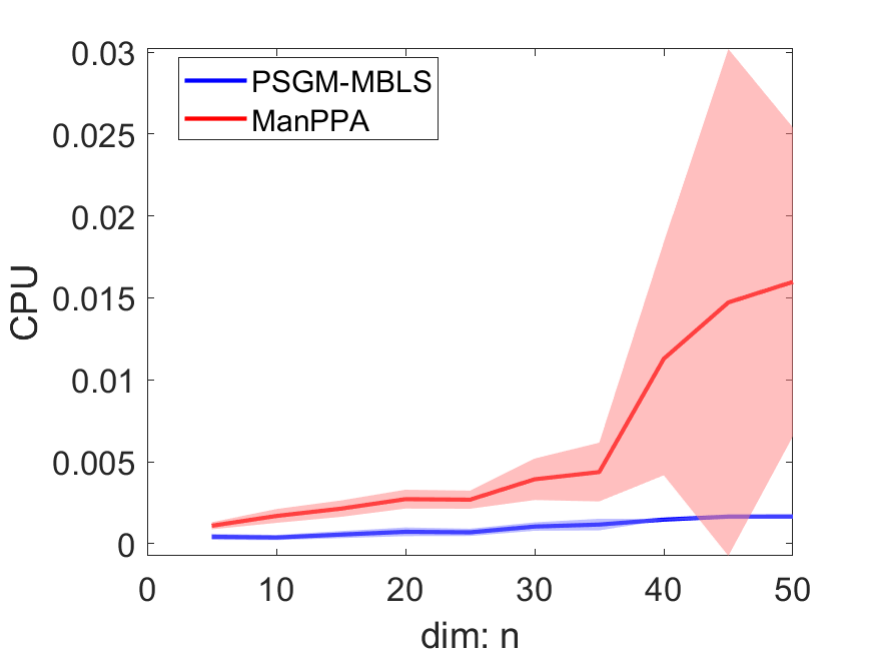}
	    \caption{Comparison on the ODL problem~\eqref{DPCP} with $\gamma=0.3$. First row: Fitting curves. Second row: CPU time versus the number of samples $p$, ($n=30$). Third row: CPU time versus the number of dimension $n$, ($p=300$). {The shadow area corresponds to the std and the line within the shadow is the mean of 10 random trials.}}
		\label{fig:DL-fit curve-single-03}
	\end{center}
\end{figure}

{\bf Matrix case.} To find the entire orthogonal basis, we use sequential ManPPA to solve problem \eqref{DPCP-matrix} with $q=n$. We compared sequential ManPPA with PSGM-MBLS (applied to \eqref{prob:deflation}) and SLP based on ALP, and report the results in Figures \ref{fig:DL-fit curve-whole-01} and \ref{fig:DL-fit curve-whole-03}. We see that sequential ManPPA is not as robust as ALP, but it is much faster than ALP. Note that there is nothing for IRLS to do, as it tackles the objective function $\| \bm{Y}^\top\bm{X} \|_{1,2}$, which is a constant when $\bm{X}^\top\bm{X}=\bm{I}_n$. {We also find that the fitting curve of PSGM-MBLS is very high, which indicates that it fails to recover the dictionary in many cases. Again, this is due to the high sensitivity of PSGM-MBLS to the choice of step size.}

\begin{figure}[!htb]
	\begin{center}
		\includegraphics[width=0.42\linewidth]{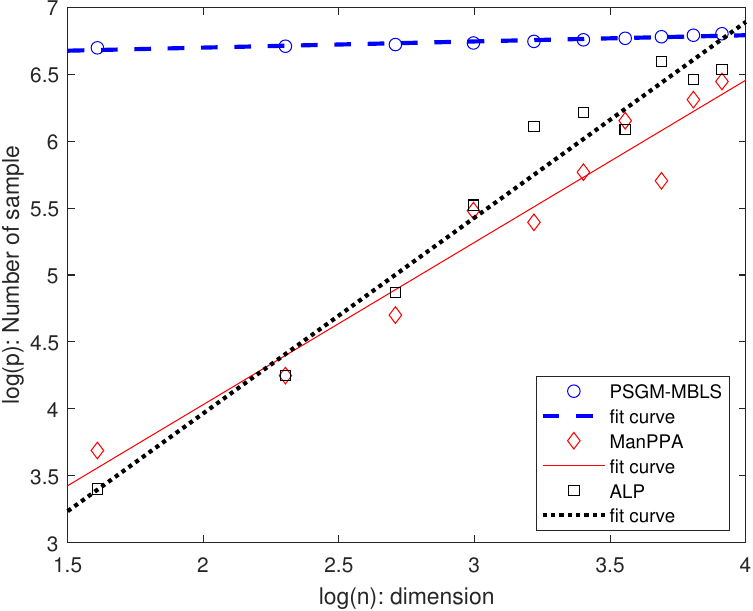} \\
		\includegraphics[width=0.38\linewidth]{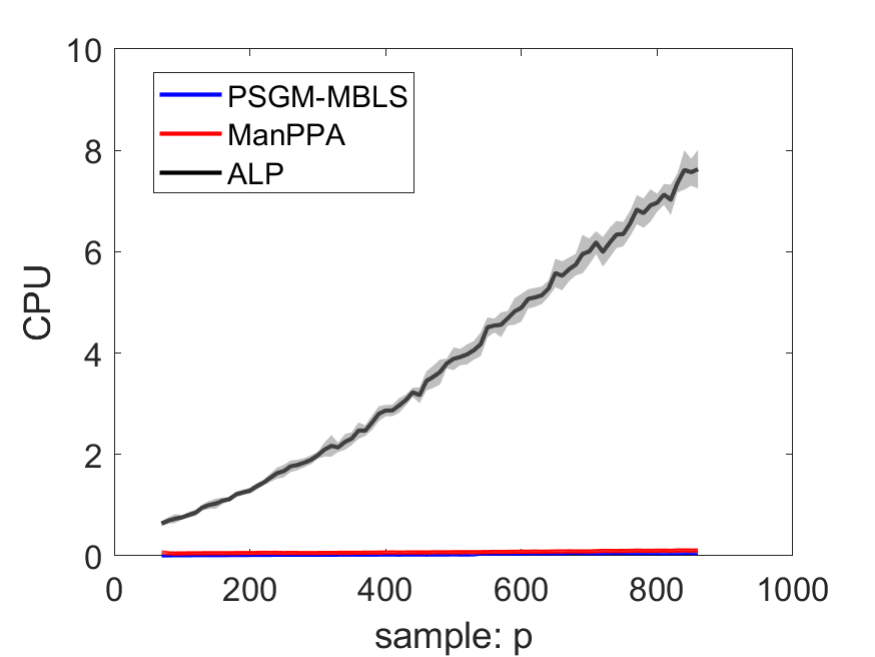}
				\includegraphics[width=0.38\linewidth]{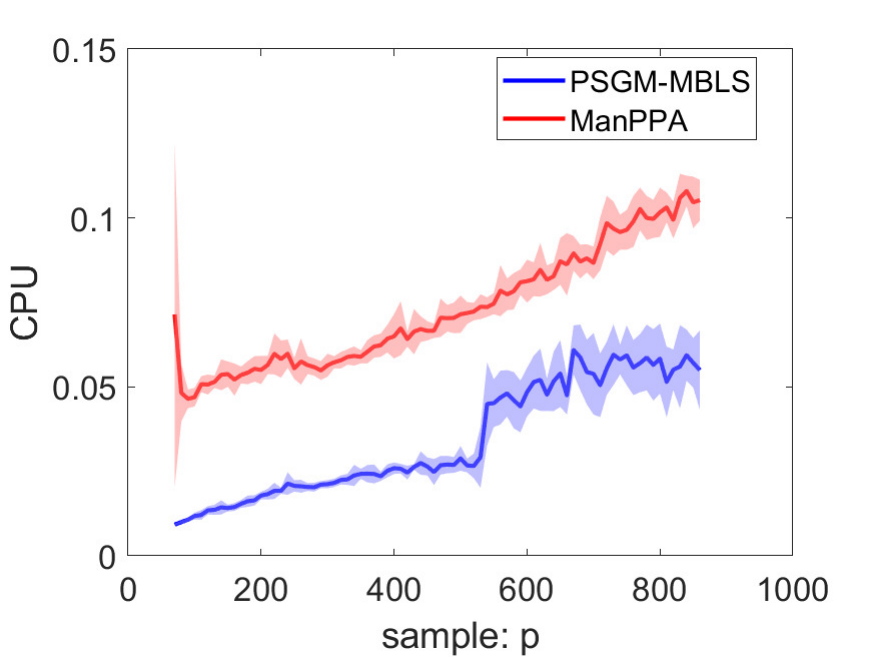}
		\includegraphics[width=0.38\linewidth]{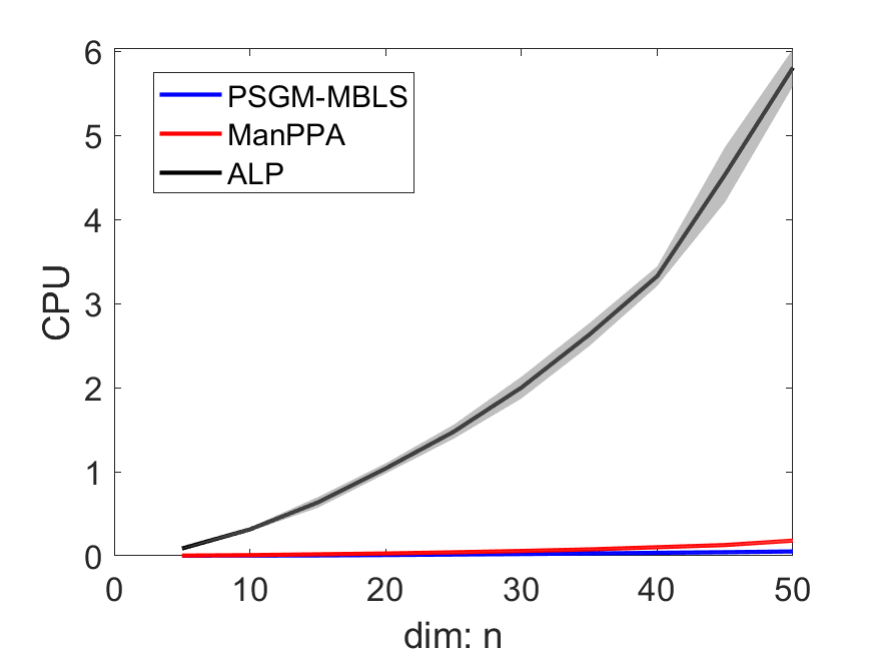}
				\includegraphics[width=0.38\linewidth]{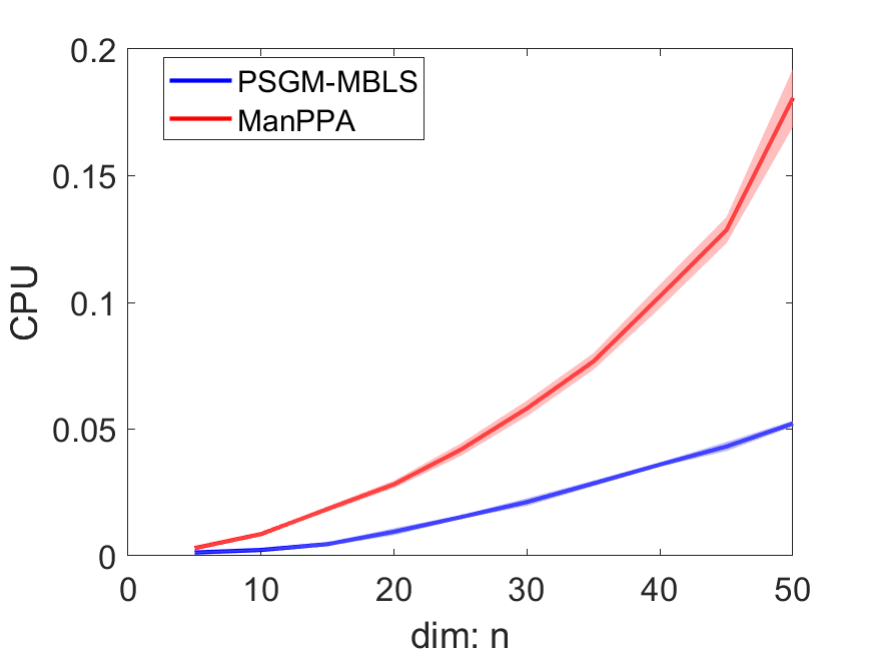}
		\caption{Comparison on the ODL problem \eqref{DPCP-matrix} with $\gamma=0.1$. First row: Linear fitting curves. Second row: CPU time versus the number of samples $p$, ($n=30$). Third row: CPU time versus the number of dimension $n$, ($p=300$). {The shadow area corresponds to the std and the line within the shadow is the mean of 10 random trials.}}
		\label{fig:DL-fit curve-whole-01}
		\end{center}
\end{figure}
		
\begin{figure}[!htb]
	\begin{center}
		\includegraphics[width=0.42\linewidth]{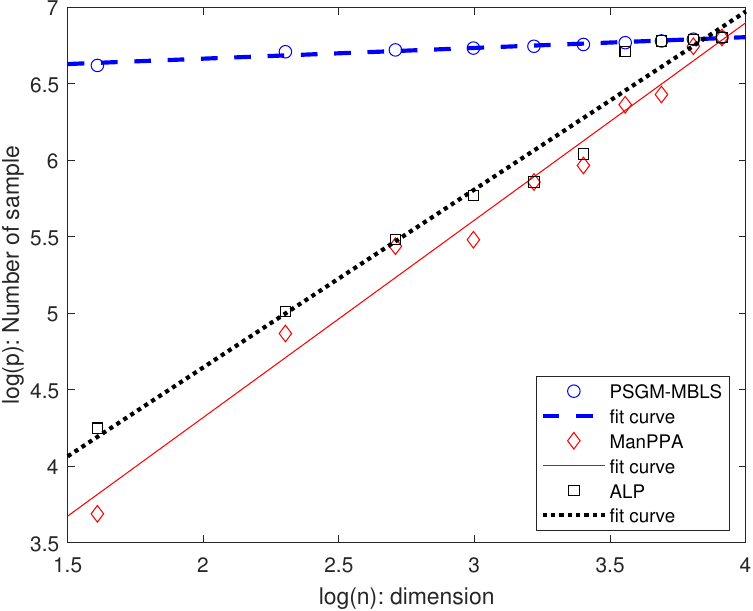} \\
		\includegraphics[width=0.38\linewidth]{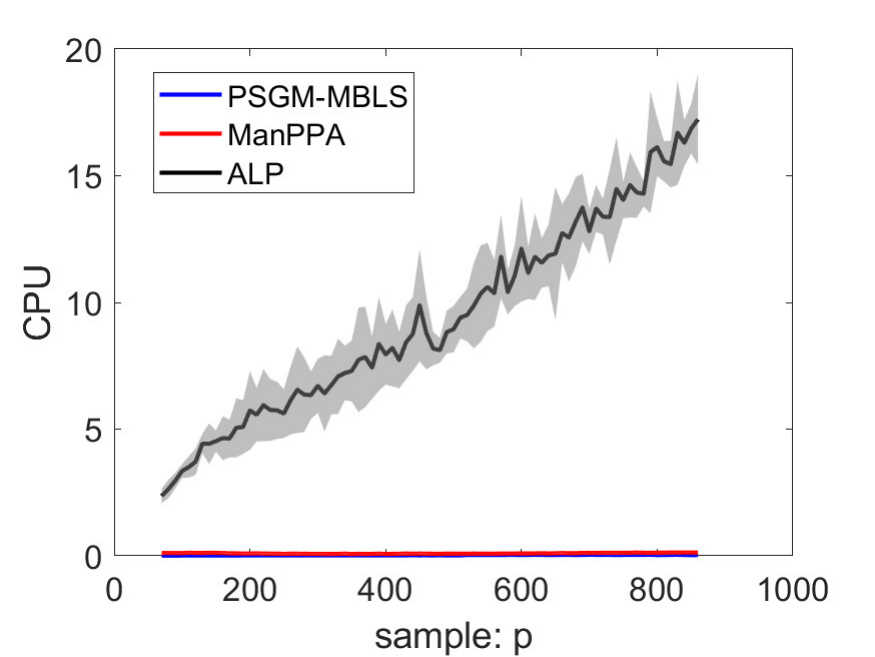}
				\includegraphics[width=0.38\linewidth]{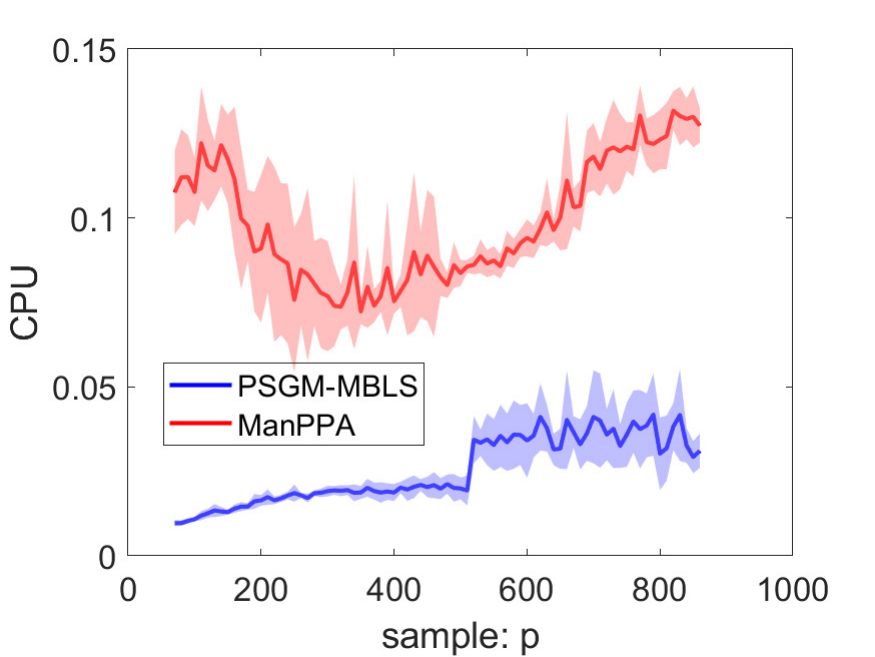}
		\includegraphics[width=0.38\linewidth]{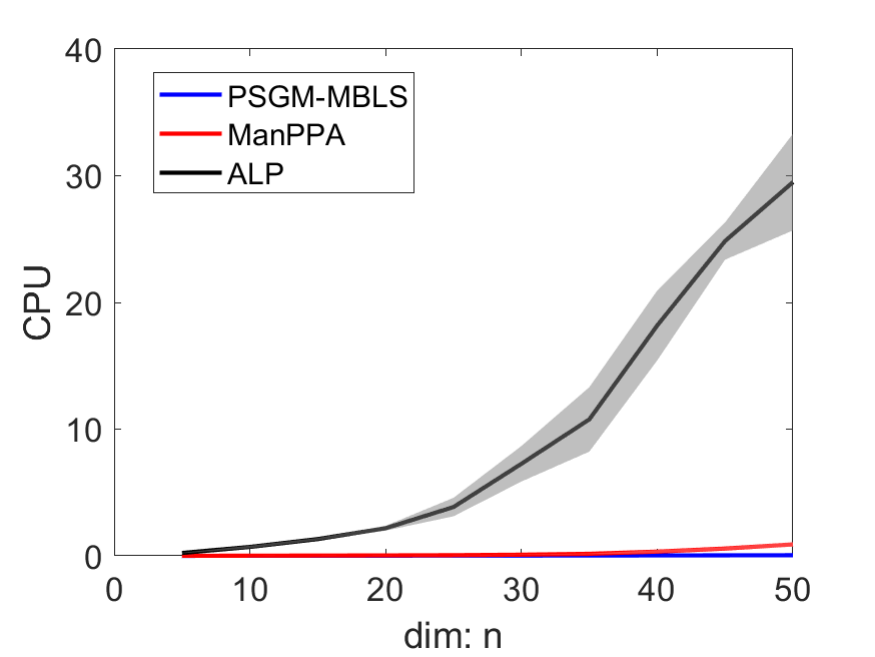}
				\includegraphics[width=0.38\linewidth]{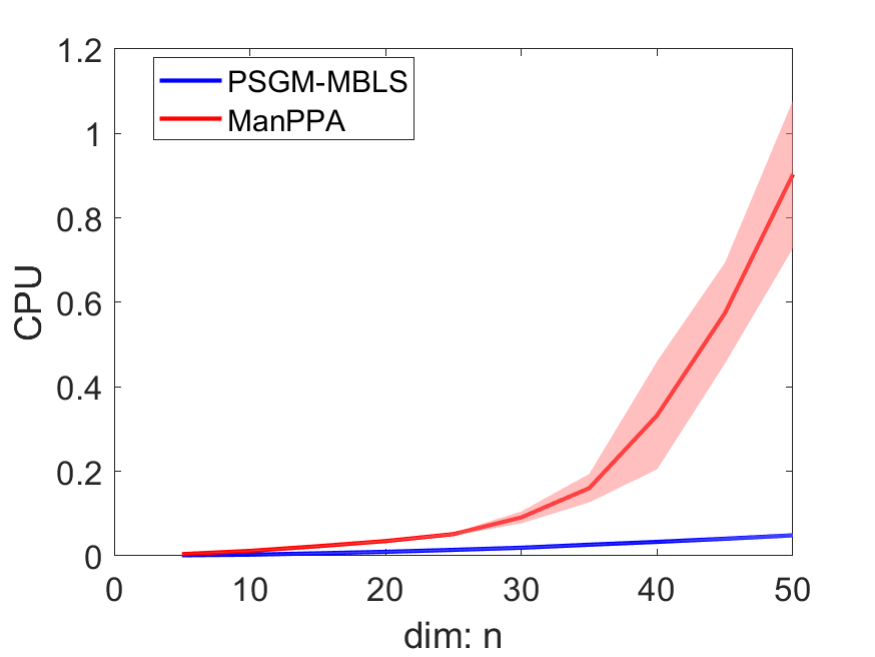}
		\caption{Comparison on the ODL problem \eqref{DPCP-matrix} with $\gamma=0.3$. First row: Linear fitting curves. Second row: CPU time versus the number of samples $p$, ($n=30$). Third row: CPU time versus the number of dimension $n$, ($p=300$). {The shadow area corresponds to the std and the line within the shadow is the mean of 10 random trials.}}
		\label{fig:DL-fit curve-whole-03}
	\end{center}
\end{figure}

\section{Conclusions}\label{sec:conclusion}
In this paper, we presented ManPPA and its stochastic variant StManPPA for solving problem~\eqref{DPCP}. By exploiting the manifold structure of the constraint set $\M$, these methods not only are practically efficient but also possess convergence guarantees that are provably superior to those of existing subgradient-type methods. Using ManPPA as a building block, we also proposed a new sequential approach to solving the matrix analog~\eqref{DPCP-matrix} of problem~\eqref{DPCP}. We conducted extensive numerical experiments to compare the performance of our proposed algorithms with existing ones on the ODL problem and DPCP formulation of the RSR problem. The results demonstrated the efficiency and efficacy of our proposed methods.

\appendix
\section*{Appendix}

\section{Useful Properties of $\Proj_{\M}$}
In this section, we collect some useful properties of the projector $\Proj_{\M}$. 
\begin{proposition} \label{prop:retr}
For any $\bm{x}\in\M$ and $\bm{d}\in\R^n$ satisfying $\bm{d}^\top\bm{x}=0$ (i.e., $\bm{d}$ is a tangent vector at $\bm{x}$), we have
\begin{equation} \label{eq:retr1}
\| \Proj_{\M}(\bm{x}+\bm{d}) - (\bm{x}+\bm{d}) \|_2 \le \frac{1}{2}\|\bm{d}\|_2^2.
\end{equation}
Moreover, if $\|\bm{d}\|_2 \le D$ for some $D \in (0,+\infty)$, then
\begin{equation} \label{eq:retr2}
\| \Proj_{\M}(\bm{x}+\bm{d}) - \bm{x} \|_2 \ge \frac{1}{(1+D^2)^{3/4}} \|\bm{d}\|_2.
\end{equation}
\end{proposition}
\begin{proof}
It is straightforward to verify that
\[ \left\| \frac{\bm{x}+\bm{d}}{\| \bm{x}+\bm{d} \|_2} -  (\bm{x}+\bm{d}) \right\|_2 = \sqrt{1+\|\bm{d}\|_2^2} - 1. \]
We then have~\eqref{eq:retr1} by using the fact that $\sqrt{1+x^2}-1 \le \tfrac{1}{2}x^2$ for all $x\in\R$. Similarly, since
\[ \left\| \frac{\bm{x}+\bm{d}}{\| \bm{x}+\bm{d} \|_2} -  \bm{x} \right\|_2^2 = 2\left( 1-\frac{1}{\sqrt{1+\|\bm{d}\|_2^2}} \right) \]
and $\tfrac{1}{\sqrt{1+x^2}} \le 1-\tfrac{1}{2(1+D^2)^{3/2}}x^2$ for all $x\in[0,D]$, we get~\eqref{eq:retr2}.
\end{proof}

\begin{proposition}\label{polar-nonexpensive}
For any $\bm{x},\bm{z}\in\M$ and $\bm{d}\in\R^n$ satisfying $\bm{d}^\top\bm{x}=0$, we have
\[ \| \Proj_{\M}(\bm{x}+\bm{d}) - \bm{z} \|_2 \le \|\bm{x}+\bm{d}-\bm{z}\|_2. \]
\end{proposition}
\begin{proof}
We compute
\begin{align}
& \left\| \frac{\bm{x}+\bm{d}}{\| \bm{x}+\bm{d} \|_2} - \bm{z} \right\|_2 = 2 - 2\frac{(\bm{x}+\bm{d})^\top\bm{z}}{\|\bm{x}+\bm{d}\|_2} \nonumber \\
=& 2 + 2(\bm{x}+\bm{d})^\top\bm{z} \left( 1 - \frac{1}{\|\bm{x}+\bm{d}\|_2} \right) - 2(\bm{x}+\bm{d})^\top\bm{z} \nonumber \\
\le& 2 + 2(\|\bm{x}+\bm{d}\|_2 - 1) - 2(\bm{x}+\bm{d})^\top\bm{z}  \nonumber \\
\le& \|\bm{x}\|_2^2 + \|\bm{z}\|_2^2 + \|\bm{d}\|_2^2 -  2(\bm{x}+\bm{d})^\top\bm{z} \label{eq:retr} \\
=& \|\bm{x}+\bm{d}-\bm{z}\|_2, \nonumber
\end{align}
where~\eqref{eq:retr} follows from the fact that $\|\bm{x}+\bm{d}\|_2-1 = \sqrt{1+\|\bm{d}\|_2^2} - 1 \le \tfrac{1}{2}\|\bm{d}\|_2^2$.
\end{proof}

\section{Proof of Proposition~\ref{prop:suff_decrease}} \label{app:suffdec}
Since $f$ is Lipschitz with constant $L$ and $\bm{d}^{k\top}\bm{x}^k = 0$, we have
\[ 
\left| f(\Proj_{\M}(\bm{x}^k + \alpha \bm{d}^k)) - f(\bm{x}^k + \alpha\bm{d}^k) \right| \le L \left\| \frac{\bm{x}^k+\alpha \bm{d}^k}{\| \bm{x}^k+\alpha \bm{d}^k \|_2} - (\bm{x}^k+\alpha \bm{d}^k) \right\|_2 \le \frac{\alpha^2L}{2} \|\bm{d}^k\|_2^2 
\]
by Proposition~\ref{prop:retr}. Hence, for any $\alpha \in (0,\bar{\alpha}]$, we have
\begin{subequations}
\begin{align}
& f(\Proj_{\M}(\bm{x}^k + \alpha \bm{d}^k)) \le  f(\bm{x}^k + \alpha\bm{d}^k) + \frac{\alpha^2L}{2} \|\bm{d}^k\|_2^2 \nonumber \\
\le& (1-\alpha) f(\bm{x}^k) + \alpha f(\bm{x}^k+\bm{d}^k) + \frac{\alpha^2L}{2}\|\bm{d}^k\|_2^2 \label{eq:cvx} \\
\le& f(\bm{x}^k) - \frac{\alpha}{t}\| \bm{d}^k \|_2^2 + \frac{\alpha^2L}{2} \|\bm{d}^k\|_2^2 \label{eq:subp-opt} \\
\le& f(\bm{x}^k) - \frac{\alpha}{2t}\| \bm{d}^k \|_2^2, \label{eq:step}
\end{align}
\end{subequations}
where~\eqref{eq:cvx} follows from the convexity of $f$, \eqref{eq:subp-opt} holds because the strong convexity of the objective function in subproblem~\eqref{ManPPA-dpcp-sub}, together with the optimality of $\bm{d}=\bm{d}^k$ and feasibility of $\bm{d}=\bm{0}$ for~\eqref{ManPPA-dpcp-sub}, implies that $f(\bm{x}^k+\bm{d}^k) + \tfrac{1}{t}\|\bm{d}^k\|_2^2 \le f(\bm{x}^k)$, and~\eqref{eq:step} is due to $\alpha \leq 1/(tL)$. If $t\leq 1/L$, then $\bar{\alpha}=1$. This completes the proof.

\section{Proof of Theorem~\ref{thm:local_rate}}\label{app:conv-anal-local}
We begin with two preparatory results. The first states that the restriction of the objective function $f$ in~\eqref{DPCP} on the nonconvex constraint set $\M$ satisfies a Riemannian subgradient inequality, which means that $f$ behaves almost like a convex function on $\M$.
\begin{proposition}\label{weakly-inequality}
Let $\bm{x}\in\M$ and $\bm{d}\in\R^n$ be such that $\bm{d}^\top\bm{x}=0$. Define $\bm{x}^+=\bm{x}+\bm{d}$. Then, for any $\bm{z}\in\M$ and $\bm{s} \in \partial f(\bm{x}^+)$, we have
\[
f(\bm{z}) - f(\bm{x}^+) \geq \inp{ (\bm{I}_n - \bm{x}\bm{x}^\top)\bm{s}}{\bm{z} - \bm{x}^+} - \frac{L}{2}\| \bm{z}-\bm{x} \|_2^2.
\]
\end{proposition}
\begin{proof}
Since $f$ is convex on $\R^n$, we have
\[
 f(\bm{z}) - f(\bm{x}^+) \geq \inp{\bm{s}}{\bm{z} - \bm{x}^+} = \inp{(\bm{I}_n-\bm{x}\bm{x}^\top)\bm{s}}{\bm{z} - \bm{x}^+}  +  \inp{\bm{x}\bm{x}^\top\bm{s}}{\bm{z} - \bm{x}^+}
\]
Now, observe that
\begin{subequations}
\begin{align}
\inp{\bm{x}\bm{x}^\top\bm{s}}{\bm{z} - \bm{x}^+} &= \inp{\bm{s}}{\bm{x}\bm{x}^\top(\bm{z} - (\bm{x}+\bm{d}))} \nonumber \\
&= \inp{\bm{s}}{\bm{x}(\bm{x}^\top\bm{z} - 1)} \label{eq:sphere-prop1} \\
&\geq -\frac{1}{2} \|\bm{s}\|_2 \| \bm{z}-\bm{x} \|_2^2, \label{eq:sphere-prop2}
\end{align}
\end{subequations}
where~\eqref{eq:sphere-prop1} is due to $\bm{x}^\top\bm{x}=1$ and $\bm{d}^\top\bm{x}=0$, while~\eqref{eq:sphere-prop2} follows from the fact that $|\bm{x}^\top\bm{z}-1|=\tfrac{1}{2}\|\bm{z}-\bm{x}\|_2^2$. Since $f$ is Lipschitz with constant $L$, we have $\|\bm{s}\|_2 \le L$.
\end{proof}

The second establishes a key recursion for the iterates generated by ManPPA.
\begin{proposition}\label{prop:sphere linear rate recursion}
Let $\{\bm{x}^k\}_k$ be the sequence generated by Algorithm \ref{alg:manppa} with $t\leq 1/L$. Then, for any $\bar{\bm x}\in\M$, we have 
\[
\| \bm{x}^{k+1} - \bar{\bm x} \|_2^2 \leq (1+tL)\| \bm{x}^{k} - \bar{\bm x} \|_2^2 - 2t\left(  f(\bm{x}^k) - f(\bar{\bm x}) \right) + t^2 L^2.
\]
\end{proposition}

\begin{proof}
Since $t\leq 1/L$, we have $\bm{x}^{k+1}=\Proj_{\M}(\bm{x}^k+\bm{d}^k)$ by Proposition \ref{prop:suff_decrease}. From the optimality condition of the subproblem \eqref{ManPPA-dpcp-sub}, there exists an $\bm{s}^k \in \partial  f(\bm{x}^{k} + \bm{d}^{k})$ such that
\be\label{lem-B4-proof-1}
\bm{d}^{k} = - t (\bm{I}_n - \bm{x}^k\bm{x}^{k\top}) \bm{s}^k.
\ee
Denoting ${\bm{x}^{k}}^{+} = \bm{x}^{k} + \bm{d}^{k}$, we have 
\begin{subequations}
\begin{align}
	& \| \bm{x}^{k+1} - \bar{\bm{x}} \|_2^2 =  \left\| \Proj_{\M}(\bm{x}^k+\bm{d}^k) - \bar{\bm{x}} \right\|_2^2 \nonumber \\
	\leq& \left\| \bm{x}^{k} + \bm{d}^k - \bar{\bm{x}} \right\|_2^2 \label{eq:non-exp} \\
	=& \| \bm{x}^{k} - \bar{\bm{x}} \|_2^2 + 2\inp{\bm{d}^k}{{\bm{x}^k}^+ - \bar{\bm{x}}} - \|\bm{d}^k\|_2^2 \nonumber \\	
	= &\| \bm{x}^{k} - \bar{\bm{x}} \|_2^2 -2t \inp{(\bm{I}_n - \bm{x}^k\bm{x}^{k\top}) \bm{s}^k}{{\bm{x}^{k}}^+ - \bar{\bm x}} - \| \bm{d}^k \|_2^2 \label{eq:d-char} \\
	\leq & (1+t L)\| \bm{x}^{k} - \bar{\bm{x}} \|_2^2  + 2t(f(\bar{\bm x}) - f({\bm{x}^k}^+)) - \| \bm{d}^k \|_2^2 \label{eq:wkineq} \\
	\leq & (1+t L)\| \bm{x}^{k} - \bar{\bm{x}} \|_2^2  + 2t(f(\bar{\bm x}) - f(\bm{x}^k)) + 2t L \| \bm{d}^k \|_2 - \| \bm{d}^k \|_2^2  \label{eq:f-lip} \\
	\leq & (1+t L)\| \bm{x}^{k} - \bar{\bm{x}} \|_2^2 + 2t(f(\bar{\bm x}) - f(\bm{x}^k)) + t^2L^2,  \label{eq:recur}
\end{align}
\end{subequations}
where~\eqref{eq:non-exp} follows from Proposition~\ref{polar-nonexpensive}, \eqref{eq:d-char} follows from~\eqref{lem-B4-proof-1}, \eqref{eq:wkineq} follows from Proposition~\ref{weakly-inequality}, \eqref{eq:f-lip} follows from the Lipschitz continuity of $f$, and~\eqref{eq:recur} follows from the fact that  $2tL\|\bm{d}^k\|_2 - \| \bm{d}^k \|_2^2= -(\|\bm{d}^k\|_2-tL)^2+t^2L^2\leq t^2L^2$.
\end{proof}

We are now ready to prove Theorem \ref{thm:local_rate}. We first prove \eqref{the-B5-proof-1} by induction. Let $\bm{x}^* \in \mathcal{X}$ be such that $\dist(\bm{x}^k,\mathcal{X}) = \|\bm{x}^k-\bm{x}^*\|_2$. By invoking Proposition \ref{prop:sphere linear rate recursion} with $\bar{\bm x}=\bm{x}^*$, we have
\begin{align*}
& \dist^2(\bm{x}^{k+1},\mathcal{X}) \leq \| \bm{x}^{k+1} -  \bm{x}^* \|_2^2 \\
\leq & (1+t L)\| \bm{x}^{k} -  \bm{x}^* \|_2^2 - 2t \left(  f(\bm{x}^k) - f( \bm{x}^*) \right) + t^2 L^2 \\
\leq & (1+tL) \dist^2(\bm{x}^k,\mathcal{X}) - 2 \alpha t \cdot \dist(\bm{x}^{k},\mathcal{X}) + t^2 L^2,
\end{align*}
where the last inequality follows from \eqref{def:sharp_ineq}. Consider the function $[0,\overline{\delta}] \ni s\mapsto\phi(s) = (1+tL) s^2 - 2t\alpha s+t^2L^2$. Observe that $\phi$ attains its maximum at $s=\overline{\delta}$ if $\overline{\delta} \geq \tfrac{2t\alpha}{1+tL}$. Given that $t\leq  \min \left\{ \tfrac{\overline{\delta}}{2\alpha-L \overline{\delta}}, \tfrac{2\overline{\delta} \alpha -L\overline{\delta}^2 }{L^2} \right\}$, we indeed have $\overline{\delta} \geq \tfrac{2t\alpha}{1+tL}$ and hence $\phi(s)\leq \phi(\overline{\delta})\leq \overline{\delta}^2$ for all $s\in[0,\overline{\delta}]$. In particular, we have $\dist(\bm{x}^{k+1},\mathcal{X})\leq \overline{\delta}$ whenever $ \dist(\bm{x}^{k},\mathcal{X})\leq \overline{\delta}$. This establishes \eqref{the-B5-proof-1}.
	
Next, we prove \eqref{the-B5-proof-2}. Again, let $\bm{x}^* \in \mathcal{X}$ be such that $\dist(\bm{x}^k,\mathcal{X}) = \|\bm{x}^k-\bm{x}^*\|_2$. Since $\alpha\leq L$ by~\eqref{def:sharp_ineq}, we have $\bar{\delta}\leq 1$. This implies that $\bm{x}^{k\top}\bm{x}^* = \tfrac{2-\|\bm{x}^k-\bm{x}^*\|_2^2}{2} \ge \tfrac{1}{2}$. Hence, the vector $\bar{\bm{d}}^k = \tfrac{\bm{x}^*}{\bm{x}^{k\top} \bm{x}^*} - \bm{x}^k$ is well defined and satisfies $\bar{\bm d}^{k\top}\bm{x}^k=0$ (i.e., $\bar{\bm d}^k$ is a tangent vector at $\bm{x}^k$), $\Proj_{\M}(\bm{x}^k+\bar{\bm d}^k) = \bm{x}^*$, and $\|\bar{\bm d}^k\|_2 \le \sqrt{3}$. By the strong convexity of the objective function in subproblem \eqref{ManPPA-dpcp-sub} and noting the optimality of $\bm{d}^k$ and feasibility of $\bar{\bm d}^k$ for~\eqref{ManPPA-dpcp-sub}, we have
\begin{equation} \label{ineq_them_quad_1}
f(\bm{x}^k + \bm{d}^k) + \frac{1}{2t}\| \bm{d}^k \|_2^2 + \frac{1}{2t}\| \bm{d}^k - \bar{\bm d}^k \|_2^2 \leq f(\bm{x}^k + \bar{\bm d}^k) + \frac{1}{2t} \| \bar{\bm d}^k \|_2^2. 
\end{equation}
Furthermore, by the Lipschitz continuity of $f$ and Proposition~\ref{prop:retr}, we get
\begin{subequations}
\begin{align}
f(\bm{x}^k + \bar{\bm d}^k) &\leq f(\Proj_{\M}(\bm{x}^k + \bar{\bm d}^k)) + \frac{L}{2}\normtwo{\bar{\bm d}^k}^2_2 = f(\bm{x}^*)+ \frac{L}{2}\normtwo{\bar{\bm d}^k}^2_2, \label{ineq_them_quad_2} \\
f(\bm{x}^{k+1}) &= f(\Proj_{\M}(\bm{x}^k + \bm{d}^k)) \leq f(\bm{x}^k+\bm{d}^k) + \frac{L}{2}\normtwo{\bm{d}^k}^2_2. \label{ineq_them_quad_3}  
\end{align}
\end{subequations}
Combining \eqref{ineq_them_quad_1}, \eqref{ineq_them_quad_2}, and \eqref{ineq_them_quad_3}, we have
\begin{equation} \label{e1}
f(\bm{x}^{k+1}) +\left(\frac{1}{2t}-\frac{L}{2}\right)\| \bm{d}^k \|_2^2 + \frac{1}{2t} \| \bm{d}^k - \bar{\bm d}^k \|_2^2 \le f(\bm{x}^*) + \left(\frac{L}{2} + \frac{1}{2t} \right) \normtwo{\bar{\bm d}^k}^2_2. 
\end{equation}
Since $t\leq \tfrac{\overline{\delta}}{2\alpha - L\overline{\delta}}\leq 1/L$, we have $\tfrac{1}{2t} - \tfrac{L}{2}\geq 0$. Moreover, since $\dist(\bm{x}^{k+1},\mathcal{X}) \leq \overline\delta \le \delta$, we have $f(\bm{x}^{k+1})-f(\bm{x}^*) \geq \alpha \cdot \dist(\bm{x}^{k+1},\mathcal{X})$ by~\eqref{def:sharp_ineq}. It then follows from \eqref{e1} and Proposition~\ref{prop:retr} that
\[
\alpha \cdot \dist(\bm{x}^{k+1},\mathcal{X}) \leq \left(\frac{L}{2}+ \frac{1}{2t}\right)\| \bar{\bm d}^k \|_2^2 \le 8\left(\frac{L}{2}+ \frac{1}{2t}\right) \|\bm{x}^k-\bm{x}^*\|^2_2 = 4\left(L+\frac{1}{t}\right)\dist^2(\bm{x}^k,\mathcal{X}).
\]
This completes the proof.

\section{Convergence Results for Inexact ALM and SSN} \label{app:ALM-SSN}

The convergence behavior of the inexact ALM (Algorithm~\ref{alg:augmented_vector}) solving problem~\eqref{ManPPA-dpcp-sub-rewrite} and the SSN method (Algorithm \ref{alg:SSN_vector}) for solving the nonsmooth equation~\eqref{nonsmooth-equation} can be deduced from existing results in the literature. We begin with the convergence result for the inexact ALM.
\begin{proposition} \label{convergece:alm}
Let $\{(\bm{d}^{j},\bm{u}^{j},y^{j},\bm{z}^{j})\}_j$ be the sequence generated by Algorithm \ref{alg:augmented_vector} with stopping criterion \eqref{augmented_subproblem_inexact_cond1}. Then, the sequence $\{(\bm{d}^{j},\bm{u}^j)\}_j$ is bounded and converges to the unique optimal solution to problem \eqref{ManPPA-dpcp-sub-rewrite}. Moreover, if $0 < \sigma_j \nearrow \sigma_{\infty} = \infty$ and the stopping criteria \eqref{augmented_subproblem_inexact_cond2} and \eqref{augmented_subproblem_inexact_cond3} are also used, then for all sufficiently large $j$, the sequence $\{(\bm{d}^j,\bm{u}^j,y^j,\bm{z}^j)\}_j$ converges asymptotically superlinearly to the set of KKT points of~\eqref{ManPPA-dpcp-sub-rewrite}.
\end{proposition}
\begin{proof}
Observe that the function $\bm{d}\mapsto \ell(\bm{d}) := \tfrac{1}{2}\|\bm{d}\|_2^2$ is strongly convex, self-conjugate (i.e., $\ell^*(\bm{d})=\ell(\bm{d})$), and has a Lipschitz continuous gradient. Moreover, the conjugate of the indicator function
\[ 
\bm{u} \mapsto \mathbb{I}_{\{ \|\cdot\|_{\infty} \le t \}}(\bm{u}) = \left\{
\begin{array}{c@{\quad}l}
0 & \mbox{if } \|\bm{u}\|_{\infty} \le t, \\
\noalign{\smallskip}
+\infty & \mbox{otherwise}
\end{array}
\right.
\]
is $\bm{u} \mapsto h(\bm{u}) = t\|\bm{u}\|_1$. Hence, we may write the dual of problem~\eqref{ManPPA-dpcp-sub-rewrite} as
\be\label{ManPPA-dpcp-sub-dual}
\max_{y\in\R,\ \bm{z}\in\R^p} \ -\left( \frac{1}{2}\| \bm{Y}\bm{z} + y\bm{x} \|_2^2 + \bm{c}^\top\bm{z} +\mathbb{I}_{\{\|\cdot\|_\infty \le t\}}(-\bm{z}) \right).
\ee
It is easy to verify that problem~\eqref{ManPPA-dpcp-sub-rewrite} has a unique optimal solution, and that problem~\eqref{ManPPA-dpcp-sub-dual} satisfies the Slater condition and its optimal solution set is also nonempty. It follows from~\cite[Theorem 4]{rockafellar1976augmented} that the first part of Proposition~\ref{convergece:alm} holds.

Next, let $g:\R\times\R^p\rightarrow\R$ denote the objective function in problem~\eqref{ManPPA-dpcp-sub-dual}; i.e.,
\[ g(y,\bm{z}) :=  \frac{1}{2}\| \bm{Y}\bm{z} + y\bm{x} \|_2^2 + \bm{c}^\top\bm{z} +\mathbb{I}_{\{\|\cdot\|_\infty \le t\}}(-\bm{z}). \]
Furthermore, define the mapping $\Gamma:\R^n\times\R^p\times\R\times\R^p\rightrightarrows\R^n\times\R^p\times\R^p\times\R$ associated with the primal-dual pair~\eqref{ManPPA-dpcp-sub-rewrite} and~\eqref{ManPPA-dpcp-sub-dual} by
\[ 
\Gamma(\bm{d},\bm{u};y,\bm{z}) := \left\{ (\bm{\pi}_1,\bm{\pi}_2,\bm{\pi}_3,\pi_4) \,\left|\,
\begin{array}{rcl}
 \bm{\pi}_1 &=& \bm{d} - y\bm{x} - \bm{Y}\bm{z}, \\
 \noalign{\smallskip}
 \bm{\pi}_2 &\in& t\partial\|\bm{u}\|_1 + \bm{z}, \\
 \noalign{\smallskip}
  -\bm{\pi}_3 &=& \bm{Y}^\top\bm{d} - \bm{u} + \bm{c}, \\
 \noalign{\smallskip}
 -\pi_4 &=& \bm{d}^\top\bm{x}.
\end{array}
\right. \right\}
\]
 It is easy to verify that if $(\bm{0},\bm{0},\bm{0},0) \in \Gamma(\bar{\bm d},\bar{\bm u};\bar{y},\bar{\bm z})$, then $(\bar{\bm d},\bar{\bm u})$ is optimal for problem~\eqref{ManPPA-dpcp-sub-rewrite} and $(\bar{y},\bar{\bm z})$ is optimal for problem~\eqref{ManPPA-dpcp-sub-dual}.

Now, it is well known (see, e.g., \cite[Section 4.2]{ZS17} and the references therein) that $(y,\bm{z}) \mapsto \partial g(y,\bm{z})$ is a polyhedral multifunction; i.e., its graph 
\[ {\rm gph}(\partial g) := \left\{ (y,\bm{z};s,\bm{t}) \in \R\times\R^p\times \R\times\R^p \mid (s,\bm{t}) \in \partial g(y,\bm{z}) \right\} \]
is the union of a finite collection of polyhedral convex sets. Moreover, using the fact that $\bm{s}\in\partial\|\bm{u}\|_1$ if and only if 
\[
s_i \in \left\{
\begin{array}{c@{\quad}l}
\{1\} & \mbox{if } u_i > 0, \\
\noalign{\smallskip}
[-1,1] & \mbox{if } u_i = 0, \\
\noalign{\smallskip}
\{-1\} & \mbox{if } u_i < 0,
\end{array}
\right.
\]
it can be verified that the KKT mapping $\Gamma$ is also a polyhedral multifunction. Hence, by invoking~\cite[Proposition 2]{Sun-lasso-2018},~\cite[Fact 2]{ZS17} and following the arguments in the proof of~\cite[Theorem 3.3]{Sun-lasso-2018}, we conclude that the second part of Proposition~\ref{convergece:alm} holds.
\end{proof}


Next, we have the following convergence result for the SSN method.
\begin{proposition} \label{Thr2}
The sequence $\{\bm{d}^j\}_j$ generated by Algorithm \ref{alg:SSN_vector} converges superlinearly to the unique optimal solution $\bar{\bm d}$ to the nonsmooth equation~\eqref{nonsmooth-equation}.
\end{proposition}
\begin{proof}
This follows from the fact that $\prox_{h/\sigma}$ is strongly semismooth (see, e.g.,~\cite[Definition 3.5]{Sun-lasso-2018} for the definition) and the arguments in the proof of~\cite[Theorem 3.6]{Sun-lasso-2018}.
\end{proof}

\section{Proof of Theorem \ref{prop:stochastic subgradient global rate}} \label{app:stmanppa}
Let $\hat{\bm x}^k=\mprox_{\lambda f}(\bm{x}^k) \in \M$. By definition of $e_\lambda$ in~\eqref{eq:moreau} and Proposition~\ref{polar-nonexpensive}, we have
\begin{equation}
 e_\lambda(\bm{x}^{k+1}) \le f(\hat{\bm x}^k) + \frac{1}{2\lambda} \| \hat{\bm x}^k - \bm{x}^{k+1} \|_2^2 \le f(\hat{\bm x}^k) + \frac{1}{2\lambda} \| \hat{\bm x}^k - (\bm{x}^k+\bm{d}^k) \|_2^2. \label{prop_sto_recursion_1}
\end{equation}
From the optimality condition of \eqref{subproblem_sppa}, we get
\[
\bm{d}^k \in -t_k (\bm{I}_n - \bm{x}^k\bm{x}^{k\top}) \partial f_{j_k}(\bm{x}^k + \bm{d}^k).
\]
Hence, we compute
\begin{subequations} \label{prop_sto_recursion_2}
\begin{align}
& \| \hat{\bm x}^k - (\bm{x}^{k} + \bm{d}^k) \|_2^2 \nonumber \\
=& \| \hat{\bm x}^k  - \bm{x}^{k} \|_2^2 - \|\bm{d}^k\|_2^2 - 2\inp{\hat{\bm x}^k  - \bm{x}^{k} - \bm{d}^k}{ \bm{d}^k} \nonumber \\
\le&  \| \hat{\bm x}^k  - \bm{x}^{k} \|_2^2 - \|\bm{d}^k\|_2^2 + 2t_k\left( f_{j_k}( \hat{\bm x}^k ) -f_{j_k}(\bm{x}^{k}+\bm{d}^k) + \frac{L_{j_k}}{2} \| \hat{\bm x}^k - \bm{x}^{k} \|_2^2 \right) \label{eq:st-1} \\
\le& \| \hat{\bm x}^k  - \bm{x}^{k} \|_2^2 - \|\bm{d}^k\|_2^2 + 2t_k\left( f_{j_k}( \hat{\bm x}^k ) -f_{j_k}(\bm{x}^{k}) \right) + 2t_kL_{j_k} \left( \frac{1}{2} \| \hat{\bm x}^k - \bm{x}^{k} \|_2^2 + \|\bm{d}^k\|_2 \right), \label{eq:st-2}
\end{align}
\end{subequations}
where~\eqref{eq:st-1} follows from Proposition \ref{weakly-inequality} and~\eqref{eq:st-2} is due to the Lipschitz continuity of $f_{j_k}$. Upon taking the expectation on both sides of \eqref{prop_sto_recursion_2} with respect to $j_k$ conditioned on $\bm{x}^k$,  we obtain
\begin{subequations}
\begin{align}
& \mathbb{E}\left[ \| \hat{\bm x}^k - (\bm{x}^{k} + \bm{d}^k) \|_2^2 \mid \bm{x}^k \right] \nonumber \\
\le& (1+t_k\bar{L})\| \hat{\bm x}^k - \bm{x}^{k} \|_2^2 + \frac{2t_k}{p} \left( f( \hat{\bm x}^k ) -f(\bm{x}^{k}) \right) + 2t_k\bar{L} \cdot \mathbb{E}\left[ \| \bm{d}^k \|_2 \mid \bm{x}^k \right] - \mathbb{E} \left[ \| \bm{d}^k \|_2^2 \mid \bm{x}^k \right]  \nonumber \\
\le& (1+t_k\bar{L})\| \hat{\bm x}^k - \bm{x}^{k} \|_2^2 + \frac{2t_k}{p} \left( f( \hat{\bm x}^k ) -f(\bm{x}^{k}) \right) + t_k^2\bar{L}^2 \label{eq:d-bd1} \\
=& \left( 1+t_k\bar{L}-\frac{t_k}{p\lambda} \right) \| \hat{\bm x}^k - \bm{x}^{k} \|_2^2 + \frac{2t_k}{p}(e_\lambda(\bm{x}^k) - f(\bm{x}^k)) + t_k^2\bar{L}^2 \nonumber \\
\le& \left( 1+t_k\bar{L}-\frac{t_k}{p\lambda} \right) \| \hat{\bm x}^k - \bm{x}^{k} \|_2^2 + t_k^2\bar{L}^2, \label{eq:d-bd2}
\end{align}
\end{subequations}
where~\eqref{eq:d-bd1} follows from the fact that $ \mathbb{E}\left[ \| \bm{d}^k \|_2 \mid \bm{x}^k \right] \le \sqrt{ \mathbb{E}\left[ \| \bm{d}^k \|_2^2 \mid \bm{x}^k \right]}$ and $a\sqrt{x}-x \le a^2/4$ for any $a,x\ge0$;~\eqref{eq:d-bd2} follows from the definition of $e_\lambda$. Putting~\eqref{prop_sto_recursion_1} and~\eqref{eq:d-bd2} together gives
\begin{align*}
 e_{\lambda}(\bm{x}^{k+1}) &\le f(\hat{\bm x}^k) + \frac{1}{2\lambda}\left( 1+t_k\bar{L}-\frac{t_k}{p\lambda} \right) \| \hat{\bm x}^k - \bm{x}^{k} \|_2^2 + \frac{t_k^2\bar{L}^2}{2\lambda} \\
&= e_{\lambda}(\bm{x}^k) + \frac{(\bar{L}-1/(p\lambda))t_k}{2\lambda}\| \hat{\bm x}^k - \bm{x}^{k} \|_2^2 + \frac{t_k^2\bar{L}^2}{2\lambda}.
\end{align*}
Taking expectation on both sides with respect to $\bm{x}^k$ yields
\[
\frac{(1/(p\lambda)-\bar{L})t_k}{2\lambda}\mathbb{E} \left[ \| \hat{\bm x}^k - \bm{x}^{k} \|_2^2 \right] \le \mathbb{E} \left[ e_{\lambda}(\bm{x}^k) \right] - \mathbb{E} \left[ e_\lambda(\bm{x}^{k+1}) \right] + \frac{t_k^2\bar{L}^2}{2\lambda}.
\]
Upon summing the above inequality over $k=0,1,\ldots T$ and noting that $\lambda <1/(p\bar{L})$ and $e_\lambda(\bm{z})\ge0$ for any $\bm{z}\in\R^n$, we obtain
\[
\sum_{k=0}^Tt_k\mathbb{E} \left[ \frac{1}{\lambda^2}\| \bm{x}^k - \mprox_{\lambda f}(\bm{x}^k) \|_2^2 \right] \leq \frac{2}{1/p-\lambda\bar{L}}e_{\lambda}(\bm{x}^0) + \frac{\bar{L}^2}{\lambda(1/p - \lambda\bar{L})} \sum_{k=0}^T t_k^2.
\]
Upon dividing both sides of the above inequality by $\sum_{k=0}^T t_k$ and noting that the left-hand side becomes $\mathbb{E} \left[ \Theta_{\lambda}(\bar{\bm x})^2 \right]$, the proof is complete.

\end{document}